\documentclass[a4paper,12pt]{amsart}
\usepackage{amsmath,amsthm}
\usepackage{amsfonts,amsmath,amsthm, amssymb, mathrsfs}
\usepackage{float}
\usepackage{euscript,eufrak,verbatim}
\usepackage{graphicx}
\usepackage[usenames]{color}
\usepackage[colorlinks,linkcolor=red,anchorcolor=blue,citecolor=blue]{hyperref}
\usepackage{amsmath}
\usepackage{amsthm}
\usepackage[all]{xy}
\usepackage{graphicx} 
\usepackage{amssymb} 
\usepackage{enumerate}
\usepackage{enumitem}
\usepackage{hyperref}

\newtheorem{thm}{Theorem}[section]
\newtheorem{prop}[thm]{Proposition}

\newtheorem{lem}[thm]{Lemma}

\def\N{\mathbb{N}}

\def\N{\mathbb{N}}

\def\logf{\log\frac{1}{\epsilon}}
\makeatletter 
\@addtoreset{equation}{section}
\numberwithin{equation}{section}

\title{Multifractal level sets and   metric mean dimension with potential}

\author{Tianlong Zhang$^1$, Ercai Chen$^1$ and Xiaoyao Zhou$^1$*}
\address
{1.School of Mathematical Sciences and Institute of Mathematics, Ministry of Education Key Laboratory of NSLSCS, Nanjing Normal University, Nanjing, 210023, Jiangsu, P.R.China}
\email{ztl20001007@163.com}
\email{ecchen@njnu.edu.cn}
\email{zhouxiaoyaodeyouxian@126.com}
\date{}

\begin{document}

\maketitle

\renewcommand{\thefootnote}{}
\footnote{*corresponding author.}

\begin{abstract}
	Let $(X,f)$ be a dynamical system with the specification property and $\varphi$ be continuous functions. In this paper, we establish some conditional  variational principles for the upper and lower Bowen/packing metric mean dimension with potential of  multifractal level set $K_\alpha:=\{x\in X:\lim\limits_{n\to\infty}\dfrac{1}{n}\sum\limits_{i=0}^{n-1}\varphi(f^ix)=\alpha\}.$
\end{abstract}

\section{Introduction}   
This paper contributes to the study of  conditional  variational principles for the upper and lower Bowen/packing metric mean dimension with potential of  multifractal level set.

The following two theories are the main backgrounds of the present paper:

\textbf{Variational principle: } Entropy is an important invariant which describes the complexity of dynamical systems.  The relationship between topological  entropy and measure entropy is  called variational principle, which is  an elegant formula in entropy theory and built a bridge between topological dynamical systems and ergodic theory. Ruelle introduced the concept of topological pressure by leading a potential into the system and established a variational principle for some transformations \cite{ru73}. Later, in 1976, Misiurewicz gave a simpler proof of the variational principle for all transformations \cite{mis76}. Topological pressure and its variational principles are essential components of thermodynamics.
Mean dimension  is introduced by Gromov \cite{gor99} in order to study the systems with infinite topological entropy (see \cite{ya80} for instance). Later,  Lindenstrauss and Weiss introduced the metric mean dimension \cite{lw00}, which is a metric version of the mean dimension. 
It is natural to tie the mean dimension theory to the ergodic theory. Lindenstrauss and Tsukamoto found a "certain measure metric mean dimension" from information theory and take rate distortions function as some certain measure-theoretic quantities and successfully established variational principle for mean dimension and metric mean dimension \cite{lt18,lt19}. It promotes the rapid development of mean dimension theory especially, for variational principle. Besides, in \cite{vv17}, the authors took some amount from entropy theory as candidate for "measure mean dimension" and  establised a good and easier-to-calculate variational principle for mean dimension and metric mean dimension. After that, lots of different variational principles also have been established, for example, by  Gutman and Spiewak \cite{gs10},  Tsukamoto \cite{tsu20}, Shi \cite{shi22}, Yang, Chen and Zhou \cite{ycz22}. 
It is worth pointing out that Tsukamoto introduced the notion of upper mean dimension with potential \cite{tsu20}, which encourage us to establish new variational principles for mean dimension with potential.

\textbf{Multifractal analysis: }   In 1973, Bowen \cite{bow73} introduced the concept of topological entropy for noncompact sets.  In 1984, Pesin and  Pitskel \cite{pp84}  extended it to the concept of topological pressure for noncompact sets. After that,  topological entropy and pressure are crossover with the study of  multifractal analysis.
 Multifractal analysis studies the complexity of level sets of invariant local
	quantities obtained from a dynamical system such as
Birkhoff averages, Lyapunov exponents, pointwise dimensions, local entropies $\cdots.$
In this paper, we focus	on the following framework.
Let $(X,d,f)$ be a topological dynamical system(abbr, $\rm{TDS}$), i.e. a compact space $(X,d)$ and a continuous transformation $f:X\to X$. For any continuous function $\varphi:X\to\mathbb{R}$, the space $X$ has a natural multifractal decomposition
$$X=\bigcup_{\alpha\in\mathbb{R}}K_{\alpha}\cup I_{\varphi}$$
where 
\begin{align*}
K_{\alpha}&=\left\{x\in X:\lim\limits_{n\to\infty}\frac{1}{n}\sum_{i=0}^{n}\varphi(f^ix)=\alpha\right\},\\
I_{\varphi}&=\left\{x\in X:\lim\limits_{n\to\infty}\frac{1}{n}\sum_{i=0}^{n}\varphi(f^ix)\ \text{does not exist}\right\}.
\end{align*}
In this paper, we take $K_{\alpha}$ as our main research object, the study of $I_{\varphi}$ will be given in another paper. There are abundant literature studying the Bowen topological entropy and topological pressure of the Level sets. For a system with the specification property, a variational principle between Bowen topological entropy of $K_{\alpha}$ and measure-theoretic entropy has been established by Takens and Verbitskiy \cite{tv03}. Thompson established a variational principle for  topological pressure for $K_{\alpha}$ in 2009 \cite{th09}.
Recently, Backes and Rodrigues give a contribution to the study of ergodic theoretical aspects of the metric mean dimension by present a variational principle on $K_{\alpha}$ of  a system with the specification property \cite{br23}. 
Recently, Cheng, Li and Selmi \cite{cls21} introduced the Bowen upper metric mean dimension with potential by open covers and  Yang, Chen, and Zhou \cite{ycz22} introduced a new dual notion called packing upper metric mean dimension with potential  and prove the relation between the three types of upper metric mean dimension. The aim of the paper is using these quantities to measure the complexity of $K_{\alpha}$. 
Now, we state our main result as follows:
\begin{thm}\label{thm 2.3}
	Let $(X,d,T)$ be a $\rm{TDS}$ satisfying the specification property. Let $\varphi,\psi\in C(X,\mathbb{R}),  and\ \alpha \in \mathcal L_\varphi$. Then
	\begin{align*}
	\overline{\rm{mdim}}_M^B(f,K_{\alpha},d,\psi)&=\overline{\rm{midm}}_M^P(f,K_{\alpha},d,\psi)\\
	&=\Lambda_{\varphi}\overline{\rm{mdim}}_M(f,K_{\alpha},d,\psi)\\
	&={\rm{H}}_{\varphi}\overline{\rm{mdim}}_M(f,K_{\alpha},d,\psi).\\
	\underline{\rm{mdim}}_M^B(f,K_{\alpha},d,\psi)&=\underline{\rm{midm}}_M^P(f,K_{\alpha},d,\psi)\\
	&=\Lambda_{\varphi}\underline{\rm{mdim}}_M(f,K_{\alpha},d,\psi)\\
	&={\rm{H}}_{\varphi}\underline{\rm{mdim}}_M(f,K_{\alpha},d,\psi).
	\end{align*}
\end{thm}

\section{Definitions and Statements}
 In this section, we introduce the notions of metric mean dimension with potential on subsets of $X$, some auxiliary quantities and our main result. 
 
Let $(X,d, f)$ be a topological dynamical system (abbr. $\mathrm{TDS}$), i.e., a compact metric space $(X,d)$ with a continuous transformation $f:X\to X$.
Given $n\in \mathbb N$, $x,y \in X$, the $n$-th Bowen metric $d_n$  on $X$ is defined by  $$d_n(x,y):=\max_{0\leq j\leq n-1}\limits d(f^{j}(x),f^j(y)).$$ Then  \emph{Bowen open  ball } of radius $\epsilon$ and  order $n$ in the metric $d_n$ around $x$ is   given by
$$B_n(x,\epsilon):=\{y\in X: d_n(x,y)<\epsilon\}.$$
We say that a subset $E$ can form an $(n,\epsilon)$-ball cover of $Z\subset X$ if $Z\subset \bigcup_{x\in E}B_n(x,\epsilon)$.
Denote $\#A$ to be the cardinality of a finite set $A$, $\partial B$ to be the boundary of set $B$ and $C(X,\mathbb{R})$ to be the all of continuous functions.
Denote $\mathcal{M}_f(X)$ to be the set that is consisted of all $f$-invariant Borel probability measures on $X$ and denote $\mathcal{M}_f^e(X)$ to be the set that is consisted of all $f$-ergodic invariant Borel probability measures on $X.$ 
Let  $\psi\in C(X,\mathbb R),$
$
S_n\psi(x):=\sum_{i=0}^{n-1}\psi(f^ix),
Var(\psi,\epsilon):=\max\{|\psi(x)-\psi(y)|:d(x,y)<\epsilon\}.
$

\subsection{Specification Property}
Let $(X,d,f)$ be a $\mathrm{TDS}$.   $f$ satisfies the  specification property means that for every $\epsilon>0$, there exists an interge $m=m(\epsilon)$ such that for any finite interge intervals $\{[a_j,b_j]_{j=1}^k\}$ with $a_{j+1}-b_j\ge m$ for $j\in \{1,\dots,k-1\}$ and any $x_1,\dots,x_k$ in $X$, there exists a point $x\in X$ such that
$$d_{b_j-a_j}(f^{a_j}x,x_j)<\epsilon\ for\ all\ j=1,\dots,k.$$

\subsection{Bowen metric mean dimension with potential for  subsets}

Given a set $Z\subset X, N\in \mathbb N, 0<\epsilon<1,s\in \mathbb R,$ and a potential $\psi\in C(X,\mathbb{R})$, we consider 
$$m_{N,\epsilon}(f,Z,s,d,\psi)=\inf\left\{\sum_{i\in I}\exp\left(-sn_i+S_{n_i}\psi(x_i)\cdot\Big(\log\frac{1}{\epsilon}\Big)\right)\right\},$$
where the infimum is taken over all finite or countable covers $\{B_{n_i}(x_i,\epsilon)\}_{i\in I}$ of $Z$ with $n_i \geq N$. Obviously, the limit 
$$m_{\epsilon}(f,Z,s,d,\psi)=\lim\limits_{N\to \infty}m_{N,\epsilon}(f,Z,s,d,\psi)$$
exists since $m_{N,\epsilon}(f,Z,s,d,\psi)$ is non-increasing when $N$ increases.
 $m_{\epsilon}(f,Z,s,d,\psi)$ has a critical value of parameter $s$ jumping from $\infty$ to $0$ and which is defined by
\begin{align*}
M_{\epsilon}(f,Z,d,\psi)&:=\inf\{s:m_{\epsilon}(f,Z,s,d,\psi)=0\}\\
&:=\sup\{s:m_{\epsilon}(f,Z,s,d,\psi)=\infty\}.
\end{align*}
The Bowen upper metric mean dimension of $f$ on $Z$ with potential $\psi$ is then defined as the following limit.
$$\overline{\rm{midm}}_M^B(f,Z,d,\psi)=\limsup\limits_{\epsilon \to 0}\frac{M_{\epsilon}(f,Z,d,\psi)}{\log\frac{1}{\epsilon}}.$$
Similarly, the  Bowen lower metric mean dimension of $f$ on $Z$ with potential $\psi$ is defined as 
$$\underline{\rm{midm}}_M^B(f,Z,d,\psi)=\liminf\limits_{\epsilon \to 0}\frac{M_{\epsilon}(f,Z,d,\psi)}{\log\frac{1}{\epsilon}}.$$
When $(X,d,f)$ is a $\rm{TDS}$ ,$\psi=0$ and take $Z=X$, $\overline{\rm{midm}}_M^B(f,X,d,\psi)$ is equal to the usual metric mean  introduced by Lindenstrauss and Weiss \cite{lw00}. 

Now, we give some properties for Bowen upper (lower) metric mean dimension of $f$ on $Z$ with potential $\psi$ and their intermediate quantities.
\begin{prop}
	For any $\psi\in C(X,\mathbb{R})$, $c\in \mathbb R$ and finite or countable covers $\{B_{n_i}(x_i,\epsilon)\}_{i\in I}$ of $Z$ with $n_i \geq N$, we have 
	$$\overline{\rm{midm}}_M^B(f,Z,d,\psi+c)=\overline{\rm{midm}}_M^B(f,Z,d,\psi)+c.$$
\end{prop}
\begin{proof}
For any $\psi\in C(X,\mathbb{R})$, 
$$\sum_{i\in I}\exp\left(-sn_i+S_{n_i}(\psi+c)(x_i)\Big(\log\frac{1}{\epsilon}\Big)\right)=\sum_{i\in I}\exp\left(-(s-c\log\frac{1}{\epsilon})n_i+S_{n_i}\psi(x_i)\Big(\log\frac{1}{\epsilon}\Big)\right).$$
Then, $m_{N,\epsilon}(f,Z, s, d,\psi+c)=m_{N,\epsilon}(f,Z, s-c\log\frac{1}{\epsilon}, d,\psi)$. Futhermore, we can obtain that  
\begin{align*}
M_{\epsilon}(f, Z, d, \psi+c  )&=\inf\{s:m_{\epsilon}(f,Z,s,d,\psi+c)=0\}\\
&=\inf\{s:m_{\epsilon}(f,Z,s-c\log\frac{1}{\epsilon},d,\psi)=0\}\\
&=\inf\{s'+c\log\frac{1}{\epsilon}:m_{\epsilon}(f,Z,s',d,\psi)=0\}\\
&=M_{\epsilon}(f, Z, d, \psi  )+c\log\frac{1}{\epsilon}.
\end{align*}
It follows that $\overline{\rm{midm}}_M^B(f,Z,d,\psi+c)=\overline{\rm{midm}}_M^B(f,Z,d,\psi)+c$.
\end{proof}
For any $\psi\in C(X,\mathbb{R})$, by the compactness of $X$, there exists a constant $c$ such that $f+c>0$. Therefore, we only need to focus on the Bowen metric mean dimension with potential that have positive values and $\psi$ be a non-negative function.

\begin{prop}
	\label{lem:2.1}
	If $\{A_n\}_{n=1}^{\infty}$ is a countable family of subsets of $X$, then
	$$M_{\epsilon}\left(f,\bigcup_{n=1}^{\infty}A_n,d,\psi\right)= \sup_{n\in\mathbb{N}}M_{\epsilon}(f,A_n,f,d,\psi).$$
\end{prop}
\begin{proof}
	It is obvious that $m_{\epsilon}(f,B,s,d,\psi)\leq m_{\epsilon}(f,A,s,d,\psi)$ for every $B\subset A\subset X.$ Therefore, $M_{\epsilon}(f,B,d,\psi)\leq M_{\epsilon}(f,A,d,\psi).$ This means $M_{\epsilon}\left(f,\bigcup_{n=1}^{\infty}A_n,d,\psi\right)\ge \sup_{n\in\mathbb{N}}M_{\epsilon}(f,A_n,f,d,\psi).$ For the another inequality, let $\delta>0$ and $\{C_{n,i}\}_{i=1}^{\infty}$ is a ball cover of $A_n$ such that
	$$\sum_{i=1}^{\infty}\exp\left(-sn_i+S_{n_i}\psi(x_i)\Big(\log\frac{1}{\epsilon}\Big)\right)\leq m_{N,\epsilon}(f,A_n,s,d,\psi)+\frac{\delta}{2^n}.$$
	Then $\{C_{n,i}\}_{n,i=1}^{\infty}$ is a ball cover of $\bigcup_{n=1}^{\infty}A_n$. We have 
	\begin{align*}
		m_{N,\epsilon}\left(f,\bigcup_{n=1}^{\infty}A_n,s,d,\psi\right)&\leq\sum_{n=1}^{\infty}\sum_{i=1}^{\infty}\exp\left(-sn_i+S_{n_i}\psi(x_i)\Big(\log\frac{1}{\epsilon}\Big)\right)\\
		&\leq\sum_{n=1}^{\infty}\left(m_{N,\epsilon}(f,A_n,s,d,\psi)+\frac{\delta}{2^n}\right)\\
		&=\sum_{n=1}^{\infty}m_{N,\epsilon}(f,A_n,s,d,\psi)+\delta.
	\end{align*}
Since this holds for any $\delta >0$, letting $\delta\to0,\ N\to\infty$, we have 
$$m_{\epsilon}\left(f,\bigcup_{n=1}^{\infty}A_n,s,d,\psi\right)\leq \sum_{n=1}^{\infty}m_{\epsilon}(f,A_n,s,d,\psi).$$
Let $s:=\sup_{n\in\mathbb{N}}M_{\epsilon}(f,A_n,d,\psi)$, taking $s<t$. Then $M_{\epsilon}(f,A_n,d,\psi)\leq s<t$ for all $n\in\mathbb{N}.$ Hence, $m_{\epsilon}(f,A_n,t,d,\psi)=0$ for all $n\in\mathbb{N}$. Then 
$$m_{\epsilon}\left(f,\bigcup_{n=1}^{\infty}A_n,t,d,\psi\right)\leq \sum_{n=1}^{\infty}m_{\epsilon}(f,A_n,t,d,\psi)=0.$$
Thus, $M_{\epsilon}(f,\bigcup_{i=1}^{\infty}A_n,d,\psi)\leq t$ for every $t> s$. Letting $t\to s$, we have
$$M_{\epsilon}\left(f,\bigcup_{n=1}^{\infty}A_n,f,\psi\right)\leq \sup_{n\in\mathbb{N}}M_{\epsilon}(f,A_n,d,\psi).$$
\end{proof}

\subsection{Packing metric mean dimension with potential for  subsets}
Given a set $Z\subset X, N\in \mathbb N, 0<\epsilon<1,s\in \mathbb R,$ and a potential $\psi\in C(X,\mathbb{R})$, we consider 
$$P_{N,\epsilon}(f,Z,s,d,\psi)=\sup\left\{\sum_{i\in I}\exp\left(-sn_i+S_{n_i}\psi(x_i)\Big(\log\frac{1}{\epsilon}\Big)\right)\right\},$$
where the supremum is taken over all  finite or countable  pairwise disjoint  closed  families $\{\overline B_{n_i}(x_i,\epsilon)\}_{i\in I}$ of $Z$ with $n_i \geq N, x_i\in Z$ for all $i\in I$.
Obviously, the limit 
$$P_{\epsilon}(f,Z,s,d,\psi)=\lim\limits_{N\to \infty}P_{N,\epsilon}(f,Z,s,d,\psi)$$
exists since $P_{N,\epsilon}(f,Z,s,d,\psi)$ is non-increasing when $N$ increases.
Set$$\mathcal P_{\epsilon}(f,Z,s,d,\psi)
=\inf\left\{\sum_{i=1}^{\infty}P_{\epsilon}(f,Z,s,d,\psi): \cup_{i\geq 1}Z_i \supseteq Z \right\}.$$
$\mathcal P_{\epsilon}(f,Z,s,d,\psi)$ has a critical value of parameter $s$ jumping from $\infty$ to $0$ and which is defined by
\begin{align*}
\mathcal P_{\epsilon}(f,Z,d,\psi)&:=\inf\{s:\mathcal P_{\epsilon}(f,Z,s,d,\psi)=0\}\\
&:=\sup\{s:\mathcal P_{\epsilon}(f,Z,s,d,\psi)=\infty\}.
\end{align*}
The packing upper(lower) metric mean dimension of $f$ on $Z$ with potential $\psi$ are then defined as the following limits.
$$\overline{\rm{midm}}_M^P(f,Z,d,\psi)=\limsup\limits_{\epsilon \to 0}\frac{\mathcal P_{\epsilon}(f,Z,d,\psi)}{\log\frac{1}{\epsilon}},$$
$$\underline{\rm{midm}}_M^P(f,Z,d,\psi)=\liminf\limits_{\epsilon \to 0}\frac{\mathcal P_{\epsilon}(f,Z,d,\psi)}{\log\frac{1}{\epsilon}}.$$
We recall the definition of  packing upper metric mean dimension of $f$ on $Z$ and see that  $\overline{\rm{midm}}_M^P(f,Z,d)=\overline{\rm{midm}}_M^P(f,Z,d,0).$

\begin{prop}
	For any $\psi\in C(X,\mathbb{R})$, $c\in \mathbb R$ and finite or countable  pairwise disjoint  closed  families $\{\overline B_{n_i}(x_i,\epsilon)\}_{i\in I}$ of $Z$ with $n_i \geq N, x_i\in Z$ for all $i\in I$, we have 
	$$\overline{\rm{midm}}_M^P(f,Z,d,\psi+c)=\overline{\rm{midm}}_M^P(f,Z,d,\psi)+c.$$
\end{prop}
\begin{proof}
For any $\psi\in C(X,\mathbb{R})$, 
$$\sum_{i\in I}\exp\left(-sn_i+S_{n_i}(\psi+c)(x_i)\Big(\log\frac{1}{\epsilon}\Big)\right)=\sum_{i\in I}\exp\left(-(s-c\log\frac{1}{\epsilon})n_i+S_{n_i}\psi(x_i)\Big(\log\frac{1}{\epsilon}\Big)\right).$$
Then, $P_{N,\epsilon}(f,Z, s, d,\psi+c)=P_{N,\epsilon}(f,Z, s-c\log\frac{1}{\epsilon}, d,\psi)$. 
Futhermore, we can obtain
$$P_{\epsilon}(f,Z, s, d,\psi+c)=P_{\epsilon}(f,Z, s-c\log\frac{1}{\epsilon}, d,\psi),$$
$$\mathcal P_{\epsilon}(f,Z, s, d,\psi+c)=\mathcal P_{\epsilon}(f,Z, s-c\log\frac{1}{\epsilon}, d,\psi).$$
Then,
\begin{align*}
\mathcal P_{\epsilon}(f, Z, d, \psi+c  )&=\inf\{s:\mathcal P_{\epsilon}(f,Z,s,d,\psi+c)=0\}\\
&=\inf\{s:\mathcal P_{\epsilon}(f,Z,s-c\log\frac{1}{\epsilon},d,\psi)=0\}\\
&=\inf\{s'+c\log\frac{1}{\epsilon}:\mathcal P_{\epsilon}(f,Z,s',d,\psi)=0\}\\
&=\mathcal P_{\epsilon}(f, Z, d, \psi  )+c\log\frac{1}{\epsilon}.
\end{align*}
It follows that $\overline{\rm{midm}}_M^P(f,Z,d,\psi+c)=\overline{\rm{midm}}_M^P(f,Z,d,\psi)+c$.
\end{proof}
For any $\psi\in C(X,\mathbb{R})$, by the compactness of $X$, there exists a constant $c$ such that $f+c>0$. Therefore, we only need to focus on the packing metric mean dimension with potential that have positive values and $\psi$ be a non-negative function.

\begin{prop}\label{prop A}
	If $\{A_n\}_{n=1}^{\infty}$ is a countable family of subsets of $X$, then
	$$\mathcal P_{\epsilon}\left(f,\bigcup_{n=1}^{\infty}A_n,d,\psi\right)= \sup_{n\in\mathbb{N}}\mathcal P_{\epsilon}(f,A_n,f,d,\psi).$$
\end{prop}
\begin{proof}
	It is obvious that $\mathcal P_{\epsilon}(f,B,s,d,\psi)\leq \mathcal P_{\epsilon}(f,A,s,d,\psi)$ for every $B\subset A\subset X.$ Therefore, $\mathcal P_{\epsilon}(f,B,d,\psi)\leq \mathcal P_{\epsilon}(f,A,d,\psi).$ This means $\mathcal P_{\epsilon}\left(f,\bigcup_{n=1}^{\infty}A_n,d,\psi\right)\ge \sup_{n\in\mathbb{N}}\mathcal P_{\epsilon}(f,A_n,f,d,\psi).$ For the another inequality, let $\delta>0$ and for each $n\in\mathbb N$ choose a countable cover $\{C_{n,i}\}_{i=1}^{\infty}$  of $A_n$ such that
	$$\sum_{i=1}^{\infty}P_{\epsilon}(f,C_{n,i},s,d,\psi)\leq \mathcal P_{\epsilon}(f,A_n,s,d,\psi)+\frac{\delta}{2^n}.$$
	Then $\{C_{n,i}\}_{n,i=1}^{\infty}$ is a countable cover of $\bigcup_{n=1}^{\infty}A_n$. We have 
	\begin{align*}
		\mathcal P_{\epsilon}(f,A,s,d,\psi)&\leq\sum_{n=1}^{\infty}\sum_{i=1}^{\infty}P_{\epsilon}(f,C_{n,i},s,d,\psi)\\
		&\leq\sum_{n=1}^{\infty}\mathcal P_{\epsilon}(f,A_n,s,d,\psi)+\frac{\delta}{2^n}\\
		&=\sum_{n=1}^{\infty}\mathcal P_{\epsilon}(f,A_n,s,d,\psi)+\delta.
	\end{align*}
Since this holds for any $\delta >0$, letting $\delta\to0$, we have 
$$\mathcal P_{\epsilon}(f,A,s,d,\psi)\leq\sum_{n=1}^{\infty}\mathcal P_{\epsilon}(f,A_n,s,d,\psi).$$
Let $s:=\sup_{n\in\mathbb{N}}\mathcal P_{\epsilon}(f,A_n,d,\psi)$, taking $s<t$. Then $\mathcal P_{\epsilon}(f,A_n,d,\psi)\leq s<t$ for all $n\in\mathbb{N}.$ Hence, $\mathcal P_{\epsilon}(f,A_n,t,d,\psi)=0$ for all $n\in\mathbb{N}$. Then 
$$\mathcal P_{\epsilon}\left(f,\bigcup_{n=1}^{\infty}A_n,t,d,\psi\right)\leq \sum_{n=1}^{\infty}\mathcal P_{\epsilon}(f,A_n,t,d,\psi)=0.$$
Thus, $\mathcal P_{\epsilon}(f,\bigcup_{i=1}^{\infty}A_n,d,\psi)\leq t$ for every $t> s$. Letting $t\to s$, we have
$$\mathcal P_{\epsilon}\left(f,\bigcup_{n=1}^{\infty}A_n,f,\psi\right)\leq \sup_{n\in\mathbb{N}}\mathcal P_{\epsilon}(f,A_n,d,\psi).$$
\end{proof}

\subsection{The Auxiliary Quantities $\Lambda_{\varphi}\overline{\rm{midm}}_M(f,K_\alpha,\psi)$ and \newline
$\Lambda_{\varphi}\underline{\rm{midm}}_M(f,K_\alpha,\psi)$} $\newline$

Let $C(X,\mathbb{R})$ denote the set of all continuous functions. Take $\varphi \in C(X,\mathbb{R})$. For $\alpha \in \mathbb{R}$, let
$$K_{\alpha}:=\left\{x\in X:\lim\limits_{n\to \infty}\frac{1}{n}\sum_{j=0}^{n-1}\varphi(f^jx)=\alpha\right\}.$$
We pay attention to consider the set 
$\mathcal{L_{\varphi}}=\{\alpha \in \mathbb{R}:K_{\alpha}\ne \emptyset\}. $
Let $\delta>0,$ $\alpha\in \mathcal L_\varphi$ and $n \in \mathbb{N}$. Set
\begin{align*}
 P(\alpha,\delta,n):&=\left\{x\in X: \left| \frac{1}{n}\sum_{j=0}^{n-1}\varphi(f^j(x))-\alpha \right| <\delta\right\},\\
 N(\alpha,\delta,n,\epsilon,\psi):
 &=\inf\Big\{\sum_{x \in F}^{}\exp\left(S_n\psi(x)\Big(\log\frac{1}{\epsilon}\Big)\right):\\ &F\text{ is  an } (n,\epsilon)\text{ spanning  set  of  } P(\alpha,\delta,n)\Big\},\\
 M(\alpha,\delta,n,\epsilon,\psi):
&=\sup\Big\{\sum_{x \in E}^{}\exp\left(S_n\psi(x)\Big(\log\frac{1}{\epsilon}\Big)\right):\\ &E\text{ is  an } (n,\epsilon)\text{ separated  set  of  } P(\alpha,\delta,n)\Big\}.
\end{align*}
and  we define
$$\Lambda_{\varphi}^{\psi}(\alpha,\epsilon):=\lim\limits_{\delta \to 0}\liminf\limits_{n \to \infty}\frac{1}{n} \log N(\alpha,\delta,n,\epsilon,\psi),$$
$$\Lambda_{\varphi}\overline{\rm{midm}}_M(f,K_{\alpha},d,\psi):=\limsup\limits_{\epsilon \to 0}\frac{\Lambda_{\varphi}^{\psi}(\alpha,\epsilon)}{\logf},$$
$$\Lambda_{\varphi}\underline{\rm{midm}}_M(f,K_{\alpha},d,\psi):=\liminf\limits_{\epsilon \to 0}\frac{\Lambda_{\varphi}^{\psi}(\alpha,\epsilon)}{\log\frac{1}{\epsilon}}.$$
When $(X,d,f)$ is a $\rm{TDS}$, $\psi=0$, $\Lambda_{\varphi}\overline{\rm{midm}}_M(f,K_{\alpha},d,\psi)$ is equal to $\Lambda_{\varphi}\overline{\rm{midm}}_M(f,\alpha,d)$  introduced by Backes and Rodrigues \cite{br23}. 
\begin{prop} 
	\label{prop:2.2}
	Let $n\in\mathbb N, 0<\epsilon<1.$ Then there exists $C>0$ such that
$$N(\alpha,\delta,n,\epsilon,\psi) \leq M(\alpha,\delta,n,\epsilon,\psi) \leq N(\alpha,\delta,n,\epsilon/2,\psi)\cdot\exp(Cn).$$
\end{prop}
\begin{proof}
	If $E$ is an $(n,\epsilon)$-separated set of the maximum cardinality of $P(\alpha,\delta,n)$, then   $P(\alpha,\delta,n)\subset\bigcup_{x\in E}B_n(x,\epsilon).$  Thus, 
	$$N(\alpha,\delta,n,\epsilon,\psi) \leq M(\alpha,\delta,n,\epsilon,\psi).$$
Next, we turn to the second  inequality. Let $F$ be  an $(n,\epsilon/2)$ ball cover of $P(\alpha,\delta,n)$. Define $\Phi:E\to F$ by choosing for each $x\in E$, some point $\Phi(x)\in F$ with $x\in B_n(\Phi(x),\epsilon/2).$ Thus, $\Phi$ is injective otherwise it will contradicts with the fact that $E$ is an $(n,\epsilon)$-separated set. Moreover,
	\begin{align*}
&\sum_{y \in F}^{}\exp\left(S_n\psi(y)\cdot\log{\frac{2}{\epsilon}}\right)\\\ge&\sum_{y \in \Phi(E)}^{}\exp\left(S_n\psi(y)\cdot\log{\frac{2}{\epsilon}}\right)\\
=&\sum_{x \in E}^{}\exp\left(S_n\psi(\Phi(x))\cdot\log{\frac{2}{\epsilon}}\right)\\
	=&\sum_{x \in E}^{}\left\{\exp\left(S_n\psi(\Phi(x))\cdot\log{\frac{1}{\epsilon}}\right)\cdot \exp\left(S_n\psi(\Phi(x))\cdot\log{2}\right)\right\}\\
	\ge&\exp\left\{-n\|\psi\|\log{2}\right\}\cdot\sum_{x \in E}^{}\exp\left(S_n\psi(\Phi(x))\cdot\log{\frac{1}{\epsilon}}\right),
\end{align*}
and
\begin{align*}
	&\sum_{x \in E}^{}\exp\left(S_n\psi(\Phi(x))\cdot\log{\frac{1}{\epsilon}}\right)\\
	\ge&\exp\left\{-nVar(\psi,\epsilon/2)\Big(\log\frac{1}{\epsilon}\Big)\right\}\cdot\sum_{x \in E}^{}\exp\left(S_n\psi(x)\cdot\log{\frac{1}{\epsilon}}\right).
\end{align*}
Combining these inequalities, we have 
\begin{align*}
	&\sum_{y \in F}^{}\exp\left(S_n\psi(y)\cdot\log{\frac{2}{\epsilon}}\right)\\
	\ge&\exp\left\{-n\left(\|\psi\|\log{2}+Var(\psi,\epsilon/2)\Big(\log\frac{1}{\epsilon}\Big)\right)\right\}\cdot\sum_{x \in E}^{}\exp\left(S_n\psi(x)\cdot\log{\frac{1}{\epsilon}}\right).
	\end{align*}
This means that there exists $C>0$ satisfying
$$M(\alpha,\delta,n,\epsilon,\psi) \leq N(\alpha,\delta,n,\epsilon/2,\psi)\cdot\exp(Cn).$$
\end{proof}
 Let 
\begin{equation}\label{eq.2.1}
\Gamma_{\varphi}^{\psi}(\alpha,\epsilon):=\lim\limits_{\delta \to 0}\liminf\limits_{n \to \infty}\frac{1}{n} \log M(\alpha,\delta,n,\epsilon,\psi) .
\end{equation}
Then we have
\begin{equation}\label{eq.(2.2)}
\Lambda_{\varphi}\overline{\rm{mdim}}_M(f,K_{\alpha},d,\psi)=\limsup\limits_{\epsilon \to 0}\frac{\Gamma_{\varphi}^{\psi}(\alpha,\epsilon)}{\log\frac{1}{\epsilon}} .
\end{equation}

\subsection{The Quantities ${\rm{H}}_{\varphi}\overline{\rm{mdim}}_M(f,K_\alpha,d,\psi)$ and ${\rm{H}}_{\varphi}\underline{\rm{mdim}}_M(f,K_\alpha,d,\psi)$} $\newline$
For given  $\alpha \in \mathcal{L}_{\varphi}$ and  $\varphi \in C(X,\mathbb{R})$, we consider 
$$\mathcal{M}_f(X,\varphi,\alpha)=\left\{\mu \in \mathcal{M}_f(X)\ \text{ and  }\int\varphi \mathrm{d}\mu=\alpha\right\}.$$ 
Let $\xi=\{B_1,\dots,B_k\}$ be a finite measurable partition of $X$, the  entropy of $\xi$ with respect to $\mu$ is given by 
$$H_{\mu}(\xi)=-\sum_{i=1}^{k}\mu(B_i)\log{\mu(B_i)}.$$
Let $\xi^n=\bigvee_{j=0}^{n-1}f^{-j}\xi$. The metric entropy of $f$ with respect to $\xi$ and  $\mu$   is represented by
$$h_{\mu}(f,\xi)=\lim\limits_{n\to\infty}\frac{1}{n}H_{\mu}(\xi^n).$$
Moreover, we define 
\begin{align*}
&{\rm{H}}_{\varphi}\overline{\rm{mdim}}_M(f,K_{\alpha},d,\psi)\\
&=\limsup\limits_{\epsilon\to 0}\frac{1}{\log\frac{1}{\epsilon}}\sup_{\mu \in \mathcal{M}_f(X,\varphi,\alpha)}\left(\inf_{diam\xi<\epsilon}h_{\mu}(f,\xi)+ \Big(\log\frac{1}{\epsilon}\Big)\int\psi\mathrm{d}\mu\right),\\
&{\rm{H}}_{\varphi}\underline{\rm{mdim}}_M(f,K_{\alpha},d,\psi)\\
&=\liminf\limits_{\epsilon\to 0}\frac{1}{\log\frac{1}{\epsilon}}\sup_{\mu \in \mathcal{M}_f(X,\varphi,\alpha)}\left(\inf_{diam\xi<\epsilon}h_{\mu}(f,\xi)+ \Big(\log\frac{1}{\epsilon}\Big)\int\psi\mathrm{d}\mu\right),
\end{align*}
where  the infimum is taken over all finite measurable partition of $X$ that $diam \ \xi<\epsilon$.
When $(X,d,f)$ is a $\rm{TDS}$ ,$\psi=0$, ${\rm{H}}_{\varphi}\overline{\rm{midm}}_M(f,K_{\alpha},d,\psi)$ is equal to ${\rm{H}}_{\varphi}\overline{\rm{midm}}_M(f,\alpha,d)$  introduced by Backes and Rodrigues \cite{br23}.

\section{Proof of Main Result}
According to \cite[Proposition 3.4]{ycz22}, for $(X,d,f)$ be a $\rm{TDS}$, $\psi\in C(X,\mathbb R)$, and any non-empty subset $Z\subset X$, we have $\overline{\rm{mdim}}_M^B(f,K_{\alpha},d,\psi)\\ 
\leq\overline{\rm{midm}}_M^P(f,Z,d,\psi).$
In this section, we prove the rest part of our main result by the following  three propositions.
 Besides, we assume that $\overline{\rm{mdim}}_M^B(f,K_{\alpha},d,\psi)$, $\overline{\rm{midm}}_M^P(f,Z,d,\psi)$, $\Lambda_{\varphi}\overline{\rm{mdim}}_M(f,K_{\alpha},d,\psi)$ and \\
 ${\rm{H}}_{\varphi}\overline{\rm{mdim}}_M(f,K_{\alpha},d,\psi)$ are all finite.

 \begin{prop} \label{prop 3.1}
Under the assumptions of Theorem \ref{thm 2.3}, we have  
$$\overline{\rm{mdim}}_M^B(f,K_{\alpha},d,\psi)\leq\Lambda_{\varphi}\overline{\rm{mdim}}_M(f,K_{\alpha},d,\psi).$$
\end{prop}

\begin{proof}
Let $\{\epsilon_j\}_{j\in \mathbb{N}}$ be a sequence of positive numbers converging to zero such that
\begin{align*}
	\overline{\rm{mdim}}_M^B(f,K_{\alpha},d,\psi)=\lim \limits_{j\to\infty}\frac{P_{\epsilon_j}(f,K_{\alpha},d,\psi)}{\log\frac{1}{\epsilon_j}}.
\end{align*}
Then
$$\limsup_{j\to\infty}\frac{\Lambda_{\varphi}^{\psi}(\alpha,\epsilon_j)}{\log\frac{1}{\epsilon_j}}\leq\limsup_{\epsilon\to\infty}\frac{\Lambda_{\varphi}^{\psi}(\alpha,\epsilon)}{\log\frac{1}{\epsilon}}=\Lambda_{\varphi}\overline{\rm{mdim}}_M(f,K_{\alpha},d,\psi).$$

Given $\delta>0$ and $k\in\N$. Set
\begin{align*}
	G(\alpha,\delta,k)&:=\bigcap_{n=k}^{\infty}P(\alpha,\delta,n)\\
	&:=\bigcap_{n=k}^{\infty}\left\{x\in X: \left| \frac{1}{n}\sum_{j=0}^{n-1}\varphi(f^j(x))-\alpha \right| <\delta\right\}.
\end{align*}
Obviously,  $K_{\alpha}\subset\bigcup_{k\in\mathbb{N}}G(\alpha,\delta,k).$
For given $k\in\mathbb{N}$ and  $n\ge k,$   $G(\alpha,\delta,k)\subset P(\alpha,\delta,n)$. There exists an   $(n,\epsilon_j)$ spanning set $E$  of $P(\alpha,\delta,n)$ satisfying
\begin{align*}
	m_{\epsilon_j}(f,G(\alpha,\delta,k),s,d,\psi)\leq\sum_{x\in E}^{}&\exp\left(-sn+S_n\psi(x)\Big(\log{\frac{1}{\epsilon_j}}\Big)\right).
\end{align*}
Set $s=s(\epsilon_j)>\Lambda_{\varphi}^{\psi}(\alpha,\epsilon_j)$ and $\gamma(\epsilon_j)=(s-\Lambda_{\varphi}^{\psi}(\alpha,\epsilon_j))/2>0$. Let $\delta_j>0$ be sufficiently small. There exists an increasing sequence $\{n_l\}_{l\in\mathbb{N}}\subset\mathbb{N}$ and spanning sets $E_l$ of $P(\alpha,\delta_j,n_l)$ such that
$$\sum_{x\in E_l}^{}\exp\left(\Big(\log{\frac{1}{\epsilon_j}}\Big)S_{n_l}\psi(x)\right)\leq\exp(n_l(\Lambda_{\varphi}^{\psi}(\alpha,\epsilon_j)+\gamma(\epsilon_j))).$$
Without loss of generality we assume that $n_1>k$ and 
\begin{align*}
	m_{\epsilon_j}(f,G(\alpha,\delta_j,k),s(\epsilon_j),d,\psi)&\leq\sum_{x\in E_l}^{}\exp\left(-sn_l+\Big(\log{\frac{1}{\epsilon_j}}\Big)S_{n_l}\psi(x)\right)\\
	&=\exp(-sn_l)\exp(n_l(\Lambda_{\varphi}^{\psi}(\alpha,\epsilon_j)+\gamma(\epsilon_j)))\\
	&=\exp(-n_l\gamma(\epsilon_j)).
\end{align*}
Since $\gamma(\epsilon_j)>0$, taking $n_l \to \infty$, we have $m_{\epsilon_j}(f,G(\alpha,\delta_j,k),s(\epsilon_j),d,\psi)=0$. As a consequence,
$$P_{\epsilon_j}(f,G(\alpha,\delta_j,k),d,\psi)\leq s(\epsilon_j).$$
According to the Lemma \ref{lem:2.1}, we have that
$$P_{\epsilon_j}(f,K_{\alpha},d,\psi)\leq\sup_{k}P_{\epsilon_j}(f,G(\alpha,\delta_j,k),d,\psi)\leq s(\epsilon_j).$$
Thus,
\begin{align*}
	\overline{\rm{mdim}}_M^B(f,K_{\alpha},d,\psi)&=\limsup_{j\to\infty}\frac{P_{\epsilon_j}(f,K_{\alpha},d,\psi)}{\log\frac{1}{\epsilon_j}}\\
	&\leq\limsup_{j\to\infty}\frac{s(\epsilon_j)}{\log\frac{1}{\epsilon_j}}\\
	&\leq\limsup_{j\to\infty}\frac{2\gamma(\epsilon_j)}{\log\frac{1}{\epsilon_j}}+\limsup_{j\to\infty}\frac{\Lambda_{\varphi}^{\psi}(\alpha,\epsilon_j)}{\log{\frac{1}{\epsilon_j}}}\\
	&\leq\limsup_{j\to\infty}\frac{2\gamma(\epsilon_j)}{\log\frac{1}{\epsilon_j}}+\Lambda_{\varphi}\overline{\rm{mdim}}_M(f,K_{\alpha},d,\psi).
\end{align*}
Therefore, we can choose $s(\epsilon_j)$ sufficiently close to $\Lambda_{\varphi}^{\psi}(\alpha,\epsilon_j)$, the $\limsup$ in the last inequality  is  zero.
Hence, \begin{align*}
\overline{\rm{mdim}}_M^B(f,K_{\alpha},d,\psi)\leq\Lambda_{\varphi}\overline{\rm{mdim}}_M(f,K_{\alpha},d,\psi).
\end{align*}

\end{proof}

\begin{prop}
	Under the assumptions of Theorem  \ref{thm 2.3}, we have 
	$$\overline{\rm{mdim}}_M^P(f,K_{\alpha},d,\psi)\leq \Lambda_{\varphi}\overline{\rm{mdim}}_M(f,K_{\alpha},d,\psi).$$
\end{prop}

\begin{proof}
	Let $\{\epsilon_j\}_{j\in \mathbb{N}}$ be a sequence of positive numbers converging to zero such that
	\begin{align*}
		\overline{\rm{mdim}}_M^P(f,K_{\alpha},d,\psi)=\lim \limits_{j\to\infty}\frac{\mathcal P_{\epsilon_j}(f,K_{\alpha},d,\psi)}{\log\frac{1}{\epsilon_j}}.
	\end{align*}
	Then
	$$\limsup_{j\to\infty}\frac{\Gamma_{\varphi}^{\psi}(\alpha,\epsilon_j)}{\log\frac{1}{\epsilon_j}}\leq\limsup_{\epsilon\to\infty}\frac{\Gamma_{\varphi}^{\psi}(\alpha,\epsilon)}{\log\frac{1}{\epsilon}}=\Lambda_{\varphi}\overline{\rm{mdim}}_M(f,K_{\alpha},d,\psi).$$
	
	Given $\delta>0$ and $k\in\N$. Set
	\begin{align*}
		G(\alpha,\delta,k)&:=\bigcap_{n=k}^{\infty}P(\alpha,\delta,n)\\
		&:=\bigcap_{n=k}^{\infty}\left\{x\in X: \left| \frac{1}{n}\sum_{j=0}^{n-1}\varphi(f^j(x))-\alpha \right| <\delta\right\}.
	\end{align*}
	Obviously,  $K_{\alpha}\subset\bigcup_{k\in\mathbb{N}}G(\alpha,\delta,k).$
	For given $k\in\mathbb{N}$ and  $n\ge k,$   $G(\alpha,\delta,k)\subset P(\alpha,\delta,n)$. Then $P_{\epsilon_j}(f,G(\alpha,\delta,k),s,d,\psi)\leq P_{\epsilon_j}(f,P(\alpha,\delta,n),s,d,\psi)$ and there exists an   $(n,\epsilon_j)$-separated set $F$  of $P(\alpha,\delta,n)$ satisfying
	\begin{align*}
		P_{\epsilon_j}(f,G(\alpha,\delta,k),s,d,\psi)\leq\sum_{x\in F}^{}\exp\left(-sn+S_n\psi(x)\Big(\log{\frac{1}{\epsilon_j}}\Big)+\delta\right).
	\end{align*}
	Set $s=s(\epsilon_j)>\Gamma_{\varphi}^{\psi}(\alpha,\epsilon_j)$ and $\gamma(\epsilon_j)=(s-\Gamma_{\varphi}^{\psi}(\alpha,\epsilon_j))/2>0$. Let $\delta_j>0$ be sufficiently small. There exists an increasing sequence $\{n_l\}_{l\in\mathbb{N}}\subset\mathbb{N}$ and separated sets $F_l$ of $P(\alpha,\delta_j,n_l)$ such that
	$$\sum_{x\in F_l}^{}\exp\left(\Big(\log\frac{1}{\epsilon_j}\Big)S_{n_l}\psi(x)\right)\leq\exp(n_l(\Gamma_{\varphi}^{\psi}(\alpha,\epsilon_j)+\gamma(\epsilon_j))).$$
	Without loss of generality we assume that $n_1>k$ and 
	\begin{align*}
		P_{\epsilon_j}(f,G(\alpha,\delta_j,k),s(\epsilon_j),d,\psi)&\leq\sum_{x\in F_l}^{}\exp\left(-sn_l+\Big(\log\frac{1}{\epsilon_j}\Big)S_{n_l}\psi(x)+\delta_j\right)\\
		&=\exp(-sn_l+\delta_j)\exp(n_l(\Lambda_{\varphi}^{\psi}(\alpha,\epsilon_j)+\gamma(\epsilon_j)))\\
		&=\exp(-n_l\gamma(\epsilon_j)+\delta_j).
	\end{align*}
	Taking $n_l \to \infty$, since $\gamma(\epsilon_j)>0$, we have
	$$\mathcal P_{\epsilon_j}(f,G(\alpha,\delta_j,k),s(\epsilon_j),d,\psi)\leq P_{\epsilon_j}(f,G(\alpha,\delta_j,k),s(\epsilon_j),d,\psi)=0.$$
	As a consequence,
	$$\mathcal P_{\epsilon_j}(f,G(\alpha,\delta_j,k),d,\psi)\leq s(\epsilon_j).$$
	By Proposition \ref{prop A}, we have that
	$$\mathcal P_{\epsilon_j}(f,K_{\alpha},d,\psi)\leq \mathcal P_{\epsilon_j}\Big(f,\bigcup_{k\in\mathbb N}G(\alpha,\delta,k),d,\psi\Big)=\sup_{k}\mathcal P_{\epsilon_j}(f,G(\alpha,\delta_j,k),d,\psi).$$
	Thus,
	\begin{align*}
		\overline{\rm{mdim}}_M^P(f,K_{\alpha},d,\psi)&=\limsup_{j\to\infty}\frac{\mathcal P_{\epsilon_j}(f,K_{\alpha},d,\psi)}{\log\frac{1}{\epsilon_j}}\\
		&\leq\limsup_{j\to\infty}\frac{s(\epsilon_j)}{\log\frac{1}{\epsilon_j}}\\
		&\leq\limsup_{j\to\infty}\frac{2\gamma(\epsilon_j)}{\log\frac{1}{\epsilon_j}}+\limsup_{j\to\infty}\frac{\Gamma_{\varphi}^{\psi}(\alpha,\epsilon_j)}{\log{\frac{1}{\epsilon_j}}}\\
		&\leq\limsup_{j\to\infty}\frac{2\gamma(\epsilon_j)}{\log\frac{1}{\epsilon_j}}+\Lambda_{\varphi}\overline{\rm{mdim}}_M(f,K_{\alpha},d,\psi).
	\end{align*}
	Therefore, we can choose $s(\epsilon_j)$ sufficiently close to $\Gamma_{\varphi}^{\psi}(\alpha,\epsilon_j)$, the $\limsup$ in the last inequality  is  zero.
	Hence, \begin{align*}
	\overline{\rm{mdim}}_M^P(f,K_{\alpha},d,\psi)\leq\Lambda_{\varphi}\overline{\rm{mdim}}_M(f,K_{\alpha},d,\psi).
	\end{align*}
	
	\end{proof}

\begin{prop}
	Under the assumptions of Theorem  \ref{thm 2.3}, we have 
	$$\Lambda_{\varphi}\overline{\rm{mdim}}_M(f,K_{\alpha},d,\psi)\leq {\rm{H}}_{\varphi}\overline{\rm{mdim}}_M(f,K_{\alpha},d,\psi).$$
\end{prop}

\begin{proof}
Fix $\gamma>0$. Let $\{\epsilon_j\}_{j\in \mathbb{N}}$ be a sequence of positive numbers that converges to zero and statisfies
$$\Lambda_{\varphi}\overline{\rm{mdim}}_M(f,K_{\alpha},d,\psi)=\lim\limits_{j\to\infty}\frac{\Gamma_{\varphi}^{\psi}(\alpha,\epsilon_j)}{\log\frac{1}{\epsilon_j}}.$$ 
There exists $\epsilon_0>0$ so that for all $\epsilon_j\in \left(0,{\epsilon_0}\right]$, we have 
$$\frac{\Gamma_{\varphi}^{\psi}(\alpha,\epsilon_j)}{\log\frac{1}{\epsilon_j}}>\Lambda_{\varphi}\overline{\rm{mdim}}_M(f,K_{\alpha},d,\psi)-\frac{1}{3}\gamma.$$
Fix $\epsilon_j\in\left(0,\epsilon_0\right].$ According to (\ref{eq.2.1}), there exists a sequence of positive numbers $(\delta_{j,k})_{k\in\mathbb{N}}$ converging to zero such that for every $k\in\mathbb{N}$,
\begin{align*}
	&\left\{\liminf_{n\to\infty}\frac{1}{n}\log M(\alpha,\delta_{j,k},n,\epsilon_j,\psi)\right\}\cdot\frac{1}{\log{\frac{1}{\epsilon_j}}}\\
	&>\Lambda_{\varphi}\overline{\rm{mdim}}_M(f,K_{\alpha},d,\psi)-\frac{2}{3}\gamma.
\end{align*}
Thus, there exists a sequence of positive integers $(n_{j,k})_{k\in\mathbb{N}}$ satisfying $\lim\limits_{k\to\infty}n_{j,k}=\infty$ and  
\begin{align*}
	&\left\{\frac{1}{n_{j,k}}\log M(\alpha,\delta_{j,k},n_{j,k},\epsilon_j,\psi)\right\}\cdot\frac{1}{\log{\frac{1}{\epsilon_j}}}\\
	&>\Lambda_{\varphi}\overline{\rm{mdim}}_M(f,K_{\alpha},d,\psi)-\gamma.
\end{align*}
Let  $C_{j,k}$  be an $(n_{j,k},\epsilon_j)$ separated set of $P(\alpha,\delta_{j,k},n_{j,k})$  satisfying\begin{equation}\label{3.1}
	\left\{\frac{1}{n_{j,k}}\log{P_{j,k}}\right\}\cdot\frac{1}{\log{\frac{1}{\epsilon_j}}}>\Lambda_{\varphi}\overline{\rm{mdim}}_M(f,K_{\alpha},d,\psi)-\gamma,
\end{equation}
where
$P_{j,k}=\sum_{x\in C_{j,k}}^{}\exp\left(S_{n_{j,k}}\psi(x)\log\frac{1}{\epsilon_j}\right).$
For each $j,k\in\N,$  we construct the following  measures:
$$\sigma_k^{(j)}=\frac{1}{P_{j,k}}\sum_{x\in C_{j,k}}^{}\delta_x\exp\left(S_{n_{j,k}}\psi(x)\Big(\log{\frac{1}{\epsilon_j}}\Big)\right)$$
and
$$\mu_k^{(j)}=\frac{1}{n_{j,k}}\sum_{i=0}^{n_{j,k}-1}\sigma_k^{(j)}\circ f^{-i}.$$
It is easy to see that any accumulation point of $\{\mu_k^{(j)}\}_{k\in\mathbb{N}}$, say $\mu^{(j)}$, is $f$- invariant (see \cite[Theorem 6.9]{wal00} ). Without loss of generality, we may assume that $\lim\limits_{k\to\infty}\mu_k^{(j)}=\mu^{(j)}$. Then for $ j,k\in\mathbb{N}$, we  get
\begin{align*}
	\int\varphi \mathrm{d}{\mu_k^{(j)}}&=\int\varphi \mathrm{d}{\frac{1}{n_{j,k}}\sum_{i=0}^{n_{j,k}-1}\sigma_k^{(j)}\circ f^{-i}}=\frac{1}{n_{j,k}}\sum_{i=0}^{n_{j,k}-1}\int_{}^{}{\varphi\circ f^i} \mathrm{d}{\sigma_k^{(j)}}\\
&=\frac{1}{n_{j,k}}\sum_{i=0}^{n_{j,k}-1}\int{\varphi\circ f^i} \mathrm{d}{\frac{1}{P_{j,k}}\sum_{x\in C_{j,k}}^{}\delta_x\exp\left(S_{n_{j,k}}\psi(x)\Big(\log{\frac{1}{\epsilon_j}}\Big)\right)}\\
&=\frac{1}{n_{j,k}}\sum_{i=0}^{n_{j,k}-1}\sum_{x\in C_{j,k}}^{}\frac{\exp\left(S_{n_{j,k}}\psi(x)\Big(\log{\frac{1}{\epsilon_j}}\Big)\right)}{P_{j,k}} \varphi\circ f^i(x)\\
&=\frac{1}{P_{j,k}}\frac{1}{n_{j,k}}\sum_{x\in C_{j,k}}^{}S_{n_{j,k}}\varphi(x)\exp\left(S_{n_{j,k}}\psi(x)\Big(\log{\frac{1}{\epsilon_j}}\Big)\right)\\
&\leq\frac{1}{P_{j,k}}\frac{1}{n_{j,k}}\sum_{x\in C_{j,k}}^{}n_{j,k}(\delta_{j,k}+\alpha)\exp\left(S_{n_{j,k}}\psi(x)\Big(\log{\frac{1}{\epsilon_j}}\Big)\right)\\
&=\delta_{j,k}+\alpha.
\end{align*}
The last inequality is due to $x\in P(\alpha,\delta_{j,k},n_{j,k}).$
Similarly, we can get that $\int\varphi \mathrm{d}{\mu_k^{(j)}}\ge\alpha-\delta_{j,k}$. Thus, 
$$\left|\int\varphi \mathrm{d}{\mu_k^{(j)}}-\alpha\right|\leq\delta_{j,k}.$$
Taking $k\to\infty$, we have $\int\varphi \mathrm{d}{\mu^{(j)}}=\alpha$ for every $j\in\mathbb{N}$.
For every $j\in\mathbb{N}$, one can choose a Borel partition $\xi(j)=\{A_1,\dots,A_l\}$ of $X$ such that $diam(\xi(j))<\epsilon_j$ and $\mu^{(j)}(\partial A_i)=0$ for $1\leq i\leq l$. Then, we claim that 
$$H_{\sigma_k^{(j)}}\left(\bigvee_{i=0}^{n_{j,k}-1}f^{-i}\xi(j)\right)+\Big(\log{\frac{1}{\epsilon_j}}\Big)\int S_{n_{j,k}}\psi \mathrm{d}{\sigma_k^{(j)}}=\log P_{j,k}.$$
Now we prove the above claim. For $x$ and $y$ in the same element of $\bigvee_{i=0}^{n_{j,k}-1}f^{-i}\xi(j)$, we have $d_{n_{j,k}}(x,y)<\epsilon_j$. That means every element of $\bigvee_{i=0}^{n_{j,k}-1}f^{-i}\xi(j)$ can contain at most one point of $C_{j,k}$. Thus, we have
\begin{align*}
&H_{\sigma_k^{(j)}}\left(\bigvee_{i=0}^{n_{j,k}-1}f^{-i}\xi(j)\right)\\
=&\sum_{x\in C_{j,k}}^{}-\frac{\exp\left(S_{n_{j,k}}\psi(x)\Big(\log{\frac{1}{\epsilon_j}}\Big)\right)}{P_{j,k}}\log\frac{\exp\left(S_{n_{j,k}}\psi(x)\Big(\log{\frac{1}{\epsilon_j}}\Big)\right)}{P_{j,k}}\\
=&\sum_{x\in C_{j,k}}^{}-\frac{\exp\left(S_{n_{j,k}}\psi(x)\Big(\log{\frac{1}{\epsilon_j}}\Big)\right)}{P_{j,k}}\left(S_{n_{j,k}}\psi(x)\Big(\log{\frac{1}{\epsilon_j}}\Big)-\log P_{j,k}\right)
\end{align*}
and
\begin{align*}
	&\Big(\log{\frac{1}{\epsilon_j}}\Big)\int S_{n_{j,k}}\psi \mathrm{d}{\sigma_k^{(j)}}\\
=&\Big(\log{\frac{1}{\epsilon_j}}\Big)\int S_{n_{j,k}}\psi \mathrm{d}{\frac{1}{P_{j,k}}\sum_{x\in C_{j,k}}^{}\delta_x\exp\left(S_{n_{j,k}}\psi(x)\Big(\log{\frac{1}{\epsilon_j}}\Big)\right)}\\
=&\frac{1}{P_{j,k}}\Big(\log{\frac{1}{\epsilon_j}}\Big)\sum_{x\in C_{j,k}}^{}\exp\left(S_{n_{j,k}}\psi(x)\Big(\log{\frac{1}{\epsilon_j}}\Big)\right) S_{n_{j,k}}\psi(x).
\end{align*}
Combining the above two equalities we can get that
\begin{align*}
	&H_{\sigma_k^{(j)}}\left(\bigvee_{i=0}^{n_{j,k}-1}f^{-i}\xi(j)\right)+\Big(\log{\frac{1}{\epsilon_j}}\Big)\int S_{n_{j,k}}\psi \mathrm{d}{\sigma_k^{(j)}}\\
	=&\sum_{x\in C_{j,k}}^{}\frac{\exp\left(S_{n_{j,k}}\psi(x)\log\frac{1}{\epsilon_j}\right)}{P_{j,k}}\log P_{j,k}\\
	=&\log P_{j,k}.
\end{align*}
Fix natural number $q$ and $n_{j,k}$ with $1<q<n_{j,k}$, and for $0\leq s \leq q-1$, define $a(s)=[(n_{j,k}-s)/q]$ where $[q]$ means the biggest integer that not larger than $q$. Fix $0\leq s \leq q-1$. Then by \cite[Remark 2(ii)]{wal00}, we have 
$$\bigvee_{i=0}^{n_{j,k}-1}f^{-i}\xi(j)=\bigvee_{r=0}^{a(s)-1}f^{-(rq+s)}\left(\bigvee_{i=0}^{q-1}f^{-i}\xi(j)\right)\lor \bigvee_{t\in L}f^{-t}\xi(j),$$
where $L$ is a set with cardinality not more than $2q$. Therefore, 
$$H_{\sigma_k^{(j)}}\left(\bigvee_{i=0}^{n_{j,k}-1}f^{-i}\xi(j)\right)\leq \sum_{r=0}^{a(s)-1}H_{\sigma_k^{(j)}\circ f^{-(rq+s)}}\left(\bigvee_{i=0}^{q-1}f^{-i}\xi(j)\right)+2q\log l.$$
Then,
\begin{align*}
	\log P_{j,k}\leq &\sum_{r=0}^{a(s)-1}H_{\sigma_k^{(j)}\circ f^{-(rq+s)}}\left(\bigvee_{i=0}^{q-1}f^{-i}\xi(j)\right)+2q\log l\\
	&+\Big(\log{\frac{1}{\epsilon_j}}\Big)\int S_{n_{j,k}}\psi \mathrm{d}{\sigma_k^{(j)}}.
\end{align*}
Summing this over $s$ from 0 to $q-1$ and we have
\begin{align*}
	q\log P_{j,k}\leq &\sum_{p=0}^{n_{j,k}-1}H_{\sigma_k^{(j)}\circ f^{-p}}\left(\bigvee_{i=0}^{q-1}f^{-i}\xi(j)\right)+2q^2\log l \\
	&+q\Big(\log{\frac{1}{\epsilon_j}}\Big)\int S_{n_{j,k}}\psi \mathrm{d}{\sigma_k^{(j)}}.
\end{align*}
Now dividing by $n_{j,k}$ and using the concavity of the map $\mu \to H_{\mu}(\xi)$, we obtain
$$\frac{q}{n_{j,k}}\log P_{j,k}\leq H_{\mu_k^{(j)}}\left(\bigvee_{i=0}^{q-1}f^{-i}\xi(j)\right)+2\frac{q^2}{n_{j,k}}\log l +q\Big(\log{\frac{1}{\epsilon_j}}\Big)\int\psi \mathrm{d}\mu_k^{(j)}.$$
Then
\begin{align*}
	&q\left(\left(\Lambda_{\varphi}\overline{\rm{mdim}}_M(f,K_{\alpha},d,\psi)-\gamma\right)\log\frac{1}{\epsilon_j}\right)\overset{(3.1)}{<}\frac{q}{n_{j,k}}\log P_{j,k}\\
\leq& H_{\mu_k^{(j)}}\left(\bigvee_{i=0}^{q-1}f^{-i}\xi(j)\right)+2\frac{q^2}{n_{j,k}}\log l +q\Big(\log{\frac{1}{\epsilon_j}}\Big)\int\psi \mathrm{d}\mu_k^{(j)}.
\end{align*}
For every $B\in\bigvee_{i=0}^{q-1}f^{-i}\xi(j)$, we have $\mu^{(j)}(\partial(B))=0$. Thus, it follows from the weak convergence of measures $\mu_k^{(j)}$ to $\mu^{(j)}$ that $\lim\limits_{k\to\infty}\mu_k^{(j)}(B)=\mu^{(j)}(B)$ and, therefore,
$$\lim\limits_{k\to\infty}H_{\mu_k^{(j)}}\left(\bigvee_{i=0}^{q-1}f^{-i}\xi(j)\right)=H_{\mu^{(j)}}\left(\bigvee_{i=0}^{q-1}f^{-i}\xi(j)\right).$$
Thus taking $k\to\infty$ we have that
\begin{align*}
&q\left(\left(\Lambda_{\varphi}\overline{\rm{mdim}}_M(f,K_{\alpha},d,\psi)-\gamma\right)\log\frac{1}{\epsilon_j}\right)\\
&\leq H_{\mu^{(j)}}\left(\bigvee_{i=0}^{q-1}f^{-i}\xi(j)\right)+q\Big(\log{\frac{1}{\epsilon_j}}\Big)\int\psi d\mu^{(j)}.
\end{align*}
Dividing both side of the above inequality by $q$ and letting $q\to\infty$, we have
$$\left(\Lambda_{\varphi}\overline{\rm{mdim}}_M(f,K_{\alpha},d,\psi)-\gamma\right)\log\frac{1}{\epsilon_j}\leq h_{\mu^{(j)}}(f,\xi(j))+\Big(\log{\frac{1}{\epsilon_j}}\Big)\int\psi \mathrm{d}\mu^{(j)}$$
for all $j\in\N$, which implies that
\begin{align*}
	\Lambda_{\varphi}\overline{\rm{mdim}}_M(f,K_{\alpha},d,\psi)-\gamma\leq\frac{h_{\mu^{(j)}}(f,\xi(j))+\Big(\log{\frac{1}{\epsilon_j}}\Big)\int\psi \mathrm{d}\mu^{(j)}}{\log\frac{1}{\epsilon_j}}.
\end{align*}
Then,
\begin{align*}
	\Lambda_{\varphi}\overline{\rm{mdim}}_M(f,K_{\alpha},d,\psi)-\gamma\leq\frac{\inf_{\text{ diam }\xi<\epsilon_j}h_{\mu^{(j)}}(f,\xi)+\Big(\log{\frac{1}{\epsilon_j}}\Big)\int\psi \mathrm{d}\mu^{(j)}}{\log\frac{1}{\epsilon_j}}.
\end{align*}
As a result,
\begin{align*}
	&\Lambda_{\varphi}\overline{\rm{mdim}}_M(f,K_{\alpha},d,\psi)-\gamma\\
	&\leq\limsup_{j\to\infty}\frac{\sup_{\nu\in \mathcal{M}_f(X,\varphi,\alpha)}\inf_{\text{diam }\xi<\epsilon_j}h_{\nu}(f,\xi)+\Big(\log{\frac{1}{\epsilon_j}}\Big)\int\psi \mathrm{d}\nu }{\log\frac{1}{\epsilon_j}}\\
	&\leq\limsup_{\epsilon\to 0}\frac{\sup_{\nu\in \mathcal{M}_f(X,\varphi,\alpha)}\inf_{\text{diam }\xi<\epsilon}h_{\nu}(f,\xi)+\Big(\log{\frac{1}{\epsilon}}\Big)\int\psi \mathrm{d}\nu}{\log\frac{1}{\epsilon}}\\
	&={\rm{H}}_{\varphi}\overline{\rm{mdim}}_M(f,K_{\alpha},d,\psi)
\end{align*}
\end{proof}
Now we turn to show ${\rm{H}}_{\varphi}\overline{\rm{mdim}}_M(f,K_{\alpha},d,\psi)\leq \overline{\rm{mdim}}_M^B(f,K_{\alpha},d,\psi).$
We will construct  a Moran-like fractal   $F$ and   a special measure $\mu$, which satisfies the generalized pressure distribution principle on $F$.
Firstly, we recall some auxiliary quantities and lemmas. For $\mu \in \mathcal{M}_f(X), \ \delta\in (0,1)$ and $n\in\mathbb{N}$, denote $N_{\mu}^{\delta}(n,\epsilon)$ to be the minimal number of   $(n,\epsilon)$-balls, whose union has $\mu$-measure larger than $1-\delta$. Let
$$\overline{h}_{\mu}(f,\epsilon,\delta)=\limsup_{n\to\infty}\frac{1}{n}\log N_{\mu}^{\delta}(n,\epsilon),\ \underline{h}_{\mu}(f,\epsilon,\delta)=\liminf_{n\to\infty}\frac{1}{n}\log N_{\mu}^{\delta}(n,\epsilon).$$
For a finite Borel measurable partition $\xi$ of $X$ and a finite open cover $\mathcal{U}$ of $X$, $\xi\succ\mathcal{U}$ means that each elements of $\xi$ is contained in an element of $\mathcal{U}.$
\begin{lem}\label{lem 3.3}\cite{shi22}
Let $\mu\in \mathcal{M}_f^e(X), 0<\epsilon_2<\epsilon_1$ and $\mathcal{U}$ be a finite open over of $X$ with $diam(\mathcal{U})\leq\epsilon_1$ and $Leb(\mathcal{U})\ge\epsilon_2$. Then for any $\delta\in(0,1)$, we have
$$\overline{h}_{\mu}(f,\epsilon_1,\delta)\leq \inf_{\xi\succ\mathcal{U}}h_{\mu}(f,\xi)\leq\overline{h}_{\mu}(f,\epsilon_2,\delta),$$
$$\underline{h}_{\mu}(f,\epsilon_1,\delta)\leq \inf_{\xi\succ\mathcal{U}}h_{\mu}(f,\xi)\leq\underline{h}_{\mu}(f,\epsilon_2,\delta).$$
\end{lem}
For any $Z\subset X,\ \delta\in(0,1),\ \epsilon>0$, set
\begin{align*}
	&Q_n(Z,\psi,\epsilon):=\inf\left\{\sum_{x\in E}^{}\exp\left\{\left(\log\dfrac{1}{\epsilon} \right) S_n\psi(x)\right\}:E\ \text{is an}\ (n,\epsilon)\ \text{spanning set for}\ Z\right\}\\
	&P_n(Z,\psi,\epsilon):=\sup\left\{\sum_{x\in E}^{}\exp\left\{\left(\log\dfrac{1}{\epsilon} \right)S_n\psi(x)\right\}:E \ \text{is an}\ (n,\epsilon)\ \text{separated set for}\ Z\right\}\\
	&N_n^{\mu}(\psi,\delta,\epsilon):=\\
	&\inf\left\{\sum_{x\in E}^{}\exp\left\{\left(\log\dfrac{1}{\epsilon} \right)S_n\psi(x)\right\}:E \ \text{is an}\ (n,\epsilon)\ \text{spanning set of}\ G\subset X\ \text{with}\ \mu(G)\ge{1-\delta}\right\}\\
\end{align*}
Obviously, we have that $N_n^{\mu}(\psi,\delta,\epsilon)\leq Q_n(Z,\psi,\epsilon)\leq P_n(Z,\psi,\epsilon).$

\begin{lem}\label{lem 3.4}\cite[Proposition 1.3]{cl23}
	Let $(X,f)$ be a $\rm{TDS}$ and $\mu\in\mathcal{M}_f^e(X)$. For $\epsilon>0,\ \delta\in(0,1)$ and $\psi \in C(X,\mathbb{R})$, we have
	$$\lim\limits_{\delta\to0}\liminf_{n\to\infty}\frac{1}{n}\log{N_n^{\mu}\left(\psi,\frac{\delta}{2},\epsilon\right)}\ge\lim\limits_{\delta \to0}\underline{h}_{\mu}(f,\delta,\epsilon)+\left(\log\dfrac{1}{4\epsilon} \right)\int\psi d{\mu},$$
	$$\lim\limits_{\delta\to0}\limsup_{n\to\infty}\frac{1}{n}\log{N_n^{\mu}\left(\psi,\frac{\delta}{2},\epsilon\right)}\ge\lim\limits_{\delta \to0}\overline{h}_{\mu}(f,\delta,\epsilon)+\left(\log\dfrac{1}{4\epsilon} \right)\int\psi d{\mu}.$$
\end{lem}

\begin{prop}\label{prop 3.5}
	Under the hypotheses of theorem \ref{thm 2.3} we have that 
	$${\rm{H}}_{\varphi}\overline{\rm{mdim}}_M(f,K_{\alpha},d,\psi)\leq \overline{\rm{mdim}}_M^B(f,K_{\alpha},d,\psi).$$
	\end{prop}
\begin{proof}
Let $S:={\rm{H}}_{\varphi}\overline{\rm{mdim}}_M(f,K_{\alpha},d,\psi).$ We first consider the case  $0<S<\infty$. Given  $\gamma\in(0,\min\{S/7,1\})$ small enough and let $\{\delta_k\}_{k\in\mathbb{N}}$ be a decreasing sequence converging to $0$ with $\delta_1<\gamma/2$. There exists $\epsilon_0=\epsilon_0(\gamma)>0$ and $\mu\in\mathcal{M}_f(X,\varphi,\alpha)$ such that
\begin{gather}
	\log\frac{1}{5\epsilon_0}>1,\\\label{3.2}
	S-\gamma\leq\frac{1}{\log{\frac{1}{5\epsilon_0}}}\left(\inf_{diam\xi<5\epsilon_0}h_{\mu}(f,\xi)+\Big(\log{\frac{1}{5\epsilon_0}}\Big)\int\psi \rm{d}{\mu}\right),\\
	\sup_{\epsilon\in(0,5\epsilon_0)}\frac{M_{\epsilon}(f,K_{\alpha},d,\psi)}{\log{\frac{1}{\epsilon}}}\leq\overline{\rm{mdim}}_M^B(f,K_{\alpha},d,\psi)+\gamma.\label{3.4}
\end{gather}
Let $\mathcal{U}$ be a finite open cover of $X$ with $diam(\mathcal{U})\leq5\epsilon_0$ and $Leb(\mathcal{U})\ge\frac{5\epsilon_0}{4}$ \cite[Lemma 3.4]{shi22}. Since we can not ensure $\mu\in\mathcal{M}_f^e(X)$, we need to construct a measure which is a finite combination of some ergodic measures and is close to  $\mu$.
The following lemma is a generalized form of \cite[p. 535]{ys90}.
\begin{lem}\label{lem 3.6}
Let  $\delta_k$ and  $\epsilon>0$ be given as above.   There exists a measure $\nu_k\in\mathcal{M}_f(X)$ satisfying
	\begin{align*}
		&(1)\  \nu_k=\sum_{i=1}^{j(k)}\lambda_i\nu_i^k,\  where\ \lambda_i>0,\ \sum_{i=1}^{j(k)}\lambda_i=1\ and\ \nu_i^k\in\mathcal{M}_f^e(X);\\
		&(2)\  \inf_{\xi\succ\mathcal{U}}h_{\mu}(f,\xi)+\Big(\log{\frac{1}{\epsilon}}\Big)\int\psi \mathrm{d}{\mu}\leq\inf_{\xi\succ\mathcal{U}}h_{\nu_k}(f,\xi)+\Big(\log{\frac{1}{\epsilon}}\Big)\int\psi \mathrm{d}{\nu_k}+\delta_k;\\
		&(3)\  \left|\int\varphi \mathrm{d}{\nu_k}-\int\varphi \mathrm{d}{\mu}\right|<\delta_k.
	\end{align*}	
\end{lem}
\begin{proof}
	It is easy to see that the week*-topology on $\mathcal{M}(X)$ is metrizable, and let $d_*$ to be one of the compatible metrics. Let $\beta_k>0$ be sufficiently small such that for every $\tau_1,\tau_2\in\mathcal{M}_f(X)$, if $d_*(\tau_1,\tau_2)<\beta_k,$ then we have  
	$$\left|\int\varphi \mathrm{d}{\tau_1}-\int\varphi \mathrm{d}{\tau_2}\right|<\delta_k.$$ 
	Let $\mathcal{P}=\{P_1,\dots,P_{j(k)}\}$ be a partition of $\mathcal{M}_f(X)$ whose diameter with respect to $d_*$ is smaller than $\beta_k.$ According to the Ergodic Decomposition Theorem \cite[Remark (2)]{wal00} there exists a probability measure $\hat{\mu}$ on $\mathcal{M}_f(X)$ satisfying 
	\begin{align*}
	&	\hat{\mu}(\mathcal{M}_f^e(X))=1,\\
		&\mu=\int_{\mathcal{M}_f^e(X)}^{}\tau \mathrm{d}{\hat{\mu}(\tau)},\\
		&\int{\psi(x)\mathrm{d}{\mu(x)}}=\int_{\mathcal{M}_f^e(X)}^{}\left(\int\psi(x)\mathrm{d}{\tau(x)}\right)\mathrm{d}{\hat{\mu}(\tau)}.
	\end{align*}
	Since
	$$\sup_{\tau\in\mathcal{M}_f^e(X)}\left\{\inf_{\xi\succ\mathcal{U}}h_{\tau}(f,\xi)+\Big(\log{\frac{1}{\epsilon}}\Big)\int\psi \mathrm{d}{\tau}\right\}<\infty,$$
	there exists a $\nu_i^k\in P_i\cap{\mathcal{M}_f^e}(X)$ such that 
	$$\inf_{\xi\succ\mathcal{U}}h_{\nu_i^k}(f,\xi)+\Big(\log{\frac{1}{\epsilon_j}}\Big)\int\psi \mathrm{d}{\nu_i^k}\ge\inf_{\xi\succ\mathcal{U}}h_{\tau}(f,\xi)+\Big(\log{\frac{1}{\epsilon}}\Big)\int\psi \mathrm{d}{\tau}-\delta_k,$$ 
	for $\hat{\mu}-$almost every $\tau\in P_i\cap{\mathcal{M}_f^e}(X).$

	Let  $\lambda_i=\hat{\mu}(P_i)$ and  $\nu_k=\sum_{i=1}^{j(k)}\lambda_i\nu_i^k$. Clearly, $\nu_k$ satisfies (1) and (3). By \cite[Proposition 5]{hmry04}, we have
	$$\inf_{\xi\succ\mathcal{U}}h_{\mu}(f,\xi)=\int_{\mathcal{M}_f^e(X)}^{}\inf_{\xi\succ\mathcal{U}}h_{\tau}(f,\xi)\mathrm{d}{\hat{\mu}(\tau)}$$
	and
	$$\int\psi \mathrm{d}{\mu}=\int_{\mathcal{M}_f^e(X)}^{}\left(\int\psi(x)\mathrm{d}{\tau(x)}\right)\mathrm{d}{\hat{\mu}(\tau)}.$$
	Thus, by the choice of the measure $\nu_i^k$, we have 
	\begin{align*}
	&\inf_{\xi\succ\mathcal{U}}h_{\mu}(f,\xi)+\int\psi \mathrm{d}{\mu}\log{\frac{1}{\epsilon}}\\
	&=\sum_{i=1}^{j(k)}\int_{P_i\cap\mathcal{M}_f^e(X)}^{}\left\{\inf_{\xi\succ\mathcal{U}}h_{\tau}(f,\xi)+\Big(\log{\frac{1}{\epsilon}}\Big)\int\psi(x)\mathrm{d}{\tau(x)}\right\}d{\hat{\mu}(\tau)}\\
	&\leq\sum_{i=1}^{j(k)}\lambda_i\left\{\inf_{\xi\succ\mathcal{U}}h_{\nu_i^k}(f,\xi)+\Big(\log{\frac{1}{\epsilon}}\Big)\int\psi(x)\mathrm{d}{\nu_i^k}\right\}+\delta_k\\
	&\leq\inf_{\xi\succ\mathcal{U}}h_{\nu_k}(f,\xi)+\Big(\log{\frac{1}{\epsilon}}\Big)\int\psi \mathrm{d}{\nu_k}+\delta_k,
\end{align*}
which means that $\nu_k$ satisfies (2).
\end{proof}
Since each $\nu_i^k$ is ergodic, there exists $l_k\in\mathbb{N}$ and a  set
$$Y_{k,i}=\left\{x\in X:\left|\frac{1}{n}S_n\varphi(x)-\int\varphi \mathrm{d}{\nu_i^k}\right|<\delta_k\  \forall n\ge l_k\right\}$$
that satisfies $\nu_i^k(Y_{k,i})\ge1-\gamma$ for every $k\in\mathbb{N}$ and $i\in\{1,\dots,j(k)\}.$ 
\begin{lem}\label{lem 3.7}
	For $\epsilon_0$ and $\delta\in(0,1)$, we can find a sequence $\hat{n}_k\to\infty$ and a countable collection of finite sets $\mathcal{S}_{k,i}$ such that $\mathcal{S}_{k,i}$ is an $([\lambda_i\hat{n}_k],\frac{5\epsilon_0}{4})$ separated set for $Y_{k,i}$. We define 
	$$M_{k,i}:=\sum_{x\in\mathcal{S}_{k,i}}^{}\exp\left(S_{[\lambda_i\hat{n}_k]}\psi(x)\cdot\log{\frac{1}{5\epsilon_0}}\right).$$
	Then
	$$M_{k,i}\ge\exp\left\{[\lambda_i\hat{n}_k]\left(\inf_{\xi\succ\mathcal{U}}h_{\nu_i^k}(f,\xi)+\Big(\log\frac{1}{5\epsilon_0}\Big)\int\psi \mathrm{d}{\nu_i^k}-2\gamma\right)\right\}.$$
	in which the sequence $[\lambda_i\hat{n}_k]$ can be chosen such that $[\lambda_i\hat{n}_k]\ge l_k$ and $\hat{n}_k\ge 2^{m_k}$ where $m_k=m(\epsilon/2^{k+5})$ is as in the definition of the specification property.
\end{lem}
\begin{proof}
	From Lemma \ref{lem 3.4} and Lemma \ref{lem 3.3}, we have  
	\begin{align*}
	\liminf_{n\to\infty}\frac{1}{n}\log{N_n^{\nu_i^k}}(\psi,\delta,\frac{5\epsilon_0}{4})&\ge \underline{h}_{\nu_i^k}(f,\delta,\frac{5\epsilon_0}{4})+\Big(\log\frac{1}{5\epsilon_0}\Big)\int\psi d{\nu_i^k}-\frac{\gamma}{2}\\
&\ge \inf_{\xi\succ\mathcal{U}}h_{\nu_i^k}(f,\xi)+\Big(\log\frac{1}{5\epsilon_0}\Big)\int\psi d{\nu_i^k}-\frac{\gamma}{2}.
\end{align*}
Since $Q_n(Z,\psi,\epsilon)\leq P_n(Z,\psi,\epsilon)$ and $\nu_i^k(Y_{k,i})>1-\gamma$ for every $k$, it is immediate that
$$Q_n(Y_{k,i},\psi,\frac{5\epsilon_0}{4})\ge N_n^{\nu_i^k}(\psi,\gamma,\frac{5\epsilon_0}{4}).$$
Let $M(k,n)=P_n(Y_{k,i},\psi,\frac{5\epsilon_0}{4})$. For each $k$, we obtain 
\begin{align*}
   \liminf_{n\to\infty}\frac{1}{n}\log{M(k,n)}&\ge\liminf_{n\to\infty}\frac{1}{n}\log{N_n^{\nu_i^k}(\psi,\gamma,\frac{5\epsilon_0}{4})}\\
   &\ge \underline{h}_{\nu_i^k}(f,\frac{5\epsilon_0}{4},\gamma)+\Big(\log\frac{1}{5\epsilon_0}\Big)\int\psi d{\nu_i^k}-\frac{\gamma}{2}.
\end{align*}
Thus, we can choose a sequence $[\lambda_i\hat{n}_k]\to\infty$ as $k\to\infty$ satisfying the hypotheses of lemma so that 
$$\frac{1}{[\lambda_i\hat{n}_k]}\log{M(k,n)}\ge\underline{h}_{\nu_i^k}(f,\frac{5\epsilon_0}{4},\gamma)+\Big(\log\frac{1}{5\epsilon_0}\Big)\int\psi d{\nu_i^k}-\gamma.$$
For every $k$, let $\mathcal{S}_{k,i}$ be a $([\lambda_i\hat{n}_k],\frac{5\epsilon_0}{4})$-separated set of $Y_{k,i}$ that satisfies
$$\frac{1}{[\lambda_i\hat{n}_k]}\log{\sum_{x\in\mathcal{S}_{k,i}}^{}\exp\left(S_{[\lambda_i\hat{n}_k]}(x)\log{\frac{1}{5\epsilon_0}}\right)}\ge \frac{1}{[\lambda_i\hat{n}_k]}\log{M(k,n)}-\gamma.$$
Then we have that
\begin{align*}
	&\frac{1}{[\lambda_i\hat{n}_k]}\log{M_{k,i}}\ge \underline{h}_{\nu_i^k}(f,\frac{5\epsilon_0}{4},\gamma)+\Big(\log\frac{1}{5\epsilon_0}\Big)\int\psi d{\nu_i^k}-2\gamma, \text{i.e.,}\\
	&M_{k,i}\ge\exp\left\{[\lambda_i\hat{n}_k]\left(\inf_{\xi\succ\mathcal{U}}h_{\nu_i^k}(f,\xi)+\Big(\log\frac{1}{5\epsilon_0}\Big)\int\psi d{\nu_i^k}-2\gamma\right)\right\}.
\end{align*}
\end{proof}
For every $k$ and $y_i\in \mathcal{S}_{k,i}$, by the specification property, there exists $x=x(y_1,\dots,y_{j(k)})\in X$ that satisfies
$$d_{[\lambda_l\hat{n}_k]}(y_l,f^{a_l}x)<\frac{\epsilon_0}{2^{k+5}}$$
for $l\in\{1,\dots,j(k)\}$, where $a_1=0$ and $a_l=\sum_{i=1}^{l-1}[\lambda_i\hat{n}_k]+(l-1)m_k$ for $l\in\{2,\dots,j(k)\}.$
Let $\mathcal{S}_k$ be the set that consist of such $x=x(y_1,\dots,y_{j(k)})$. Define $n_k=\sum_{i=1}^{j(k)}[\lambda_i\hat{n}_k]+(j(k)-1)m_k$, and we have that $\frac{n_k}{\hat{n}_k}\to1$ as $k\to\infty$. 
We claim that $\mathcal{S}_k$ is a $(n_k,\frac{9\epsilon_0}{8})$ separated set and if $(y_1,\dots,y_{j(k)})\neq(y_1',\dots,y_{j(k)}')$ then $x=x(y_1,\dots,y_{j(k)})\neq x'=x'(y_1',\dots,y_{j(k)}')$. Actually, if $y_l\neq y_l'$ for some $l\in \{1,\dots,j(k)\}$ we have
\begin{align*}
	\frac{5\epsilon_0}{4}&<d_{[\lambda_l\hat{n}_k]}(y_l,y_l')\\
	&\leq d_{[\lambda_l\hat{n}_k]}(y_l,f^{a_l}x)+d_{[\lambda_l\hat{n}_k]}(y_l',f^{a_l}x')+d_{[\lambda_l\hat{n}_k]}(f^{a_l}x,f^{a_l}x')\\
	&<2\frac{\epsilon_0}{2^{k+5}}+d_{[\lambda_l\hat{n}_k]}(f^{a_l}x,f^{a_l}x').
\end{align*}
Thus,
$$d_{n_k}(x,x')\ge d_{[\lambda_l\hat{n}_k]}(f^{a_l}x,f^{a_l}x')>\frac{5\epsilon_0}{4}-\frac{\epsilon_0}{2^{k+4}}>\frac{9\epsilon_0}{8}.$$
Therefore, $\#\mathcal{S}_k=\#\mathcal{S}_{k,1}\cdots\#\mathcal{S}_{k,j(k)}.$ We define $M_k:=M_{k,1}\times \cdots\times M_{k,j(k)}.$
\begin{lem}\label{lem 3.8}
	For $k$ suffciently large, we have that
	\begin{align*}
		&(1)\ M_k\ge\exp\left\{n_k\Big(\inf_{\text{diam} \xi<5\epsilon_0}h_{\mu}(f,\xi)+\Big(\log\frac{1}{5\epsilon_0}\Big)\int\psi \mathrm{d}{\mu}-5\gamma\Big)\right\},\\
		&(2)\ \text{if}\  x\in\mathcal{S}_k,\ \left|\frac{1}{n_k}S_{n_k}\varphi(x)-\alpha\right|< 2\delta_k+Var(\varphi,\frac{\epsilon_0}{2^{k+5}})+\frac{1}{k}.
	\end{align*}
\end{lem}
\begin{proof}
	(1): for suffciently large $k$, 
	\begin{align*}
		M_k&\ge\exp\left\{\sum_{i=1}^{j(k)}[\lambda_i\hat{n}_k]\left(\inf_{\xi\succ\mathcal{U}}h_{\nu_i^k}(f,\xi)+\Big(\log\frac{1}{5\epsilon_0}\Big)\int\psi \mathrm{d}{\nu_i^k}-2\gamma\right)\right\}\\
		&=\exp\left\{\sum_{i=1}^{j(k)}\frac{[\lambda_i\hat{n}_k]}{\lambda_i}\left(\lambda_i\inf_{\xi\succ\mathcal{U}}h_{\nu_i^k}(f,\xi)+\lambda_i\Big(\log\frac{1}{5\epsilon_0}\Big)\int\psi \mathrm{d}{\nu_i^k}-2\lambda_i\gamma\right)\right\}\\
		&\ge\exp\left\{\sum_{i=1}^{j(k)}\hat{n}_k\left(\lambda_i\inf_{\xi\succ\mathcal{U}}h_{\nu_i^k}(f,\xi)+\lambda_i\Big(\log\frac{1}{5\epsilon_0}\Big)\int\psi \mathrm{d}{\nu_i^k}-3\lambda_i\gamma\right)\right\}\\
		&\ge\exp\left\{\hat{n}_k\left(\inf_{\xi\succ\mathcal{U}}h_{\mu}(f,\xi)+\Big(\log\frac{1}{5\epsilon_0}\Big)\int\psi \mathrm{d}{\mu}-3\gamma-\delta_k\right)\right\}\\
		&\ge\exp\left\{\hat{n}_k\left(\inf_{\xi\succ\mathcal{U}}h_{\mu}(f,\xi)+\Big(\log\frac{1}{5\epsilon_0}\Big)\int\psi \mathrm{d}{\mu}-\frac{7}{2}\gamma\right)\right\}\\
		&=\exp\left\{\frac{\hat{n}_k}{n_k}\left(n_k\left(\inf_{\xi\succ\mathcal{U}}h_{\mu}(f,\xi)+\Big(\log\frac{1}{5\epsilon_0}\Big)\int\psi \mathrm{d}{\mu}-\frac{7}{2}\gamma\right)\right)\right\}\\
		&\ge\exp\left\{{n_k}\left(\inf_{\xi\succ\mathcal{U}}h_{\mu}(f,\xi)+\Big(\log\frac{1}{5\epsilon_0}\Big)\int\psi \mathrm{d}{\mu}-4\gamma\right)\right\}\\
		&\ge\exp\left\{{n_k}\left(\inf_{diam\xi<5\epsilon_0}h_{\mu}(f,\xi)+\Big(\log\frac{1}{5\epsilon_0}\Big)\int\psi \mathrm{d}{\mu}-4\gamma\right)\right\}.\\
	\end{align*}
	\begin{align*}
		(2):&\left|S_{n_k}\varphi(x)-n_k\alpha\right|\\
		&\leq\sum_{i=1}^{j(k)}\left(\left|S_{[\lambda_i\hat{n}_k]}\varphi(f^{a_i}x)-S_{[\lambda_i\hat{n}_k]}\varphi(x_i)\right|+\left|S_{[\lambda_i\hat{n}_k]}\varphi(x_i)-[\lambda_i\hat{n}_k]\int\varphi \mathrm{d}{\nu_i^k}\right|\right)\\
		&+\left|\sum_{i=1}^{j(k)}\left([\lambda_i\hat{n}_k]\int\varphi\mathrm{d}{\nu_i^k}\right)-\left(\sum_{i=1}^{j(k)}[\lambda_i\hat{n}_k]\right)\alpha\right|+2(j(k)-1)m_k\|\varphi\|.
	\end{align*}
Now, we turn to estimate three parts in the right side. 
	\begin{align*}
		&\sum_{i=1}^{j(k)}\left|S_{[\lambda_i\hat{n}_k]}\varphi(f^{a_i}x)-S_{[\lambda_i\hat{n}_k]}\varphi(x_i)\right|\leq\sum_{i=1}^{j(k)}[\lambda_i\hat{n}_k]Var(\varphi,\frac{\epsilon_0}{2^{k+5}}),\\
		&\sum_{i=1}^{j(k)}\left|S_{[\lambda_i\hat{n}_k]}\varphi(x_i)-[\lambda_i\hat{n}_k]\int\varphi \mathrm{d}{\nu_i^k}\right|\leq\sum_{i=1}^{j(k)}[\lambda_i\hat{n}_k]\delta_k,\\
	\end{align*}
	\begin{align*}
		&\left|\sum_{i=1}^{j(k)}\left([\lambda_i\hat{n}_k]\int\varphi \mathrm{d}{\nu_i^k}\right)-\left(\sum_{i=1}^{j(k)}[\lambda_i\hat{n}_k]\right)\alpha\right|\\
		&\leq\left|\hat{n}_k\int\varphi \mathrm{d}{\nu^k}-\hat{n}_k\alpha\right|+2(\hat{n}_k-\sum_{i=1}^{j(k)}[\lambda_i\hat{n}_k])\|\varphi\|\leq \hat{n}_k\delta_k+2j(k)\|\varphi\|
	\end{align*}
	Thus, combining thes inequality, we have
	\begin{align*}
		&\sum_{i=1}^{j(k)}[\lambda_i\hat{n}_k]Var(\varphi,\frac{\epsilon_0}{2^{k+5}})+\sum_{i=1}^{j(k)}[\lambda_i\hat{n}_k]\delta_k+\hat{n}_k\delta_k+2j(k)\|\varphi\|\\
&\leq 2n_k\delta_k+n_kVar(\varphi,\frac{\epsilon_0}{2^{k+5}})+2j(k)\|\varphi\|
	\end{align*}
Then, for $k$ large enough, $\left|\frac{1}{n_k}S_{n_k}\varphi(x)-\alpha\right|< 2\delta_k+Var(\varphi,\frac{\epsilon_0}{2^{k+5}})+\frac{1}{k}.$
\end{proof}
Now, we begin to construct the Moran-like fractal. We choose a sequence of positive integers $\{N_k\}_{k\in\mathbb{N}}$ such that $N_1=1$ and
$$\lim\limits_{k\to\infty}\frac{n_{k+1}+m_{k+1}}{N_k}=0,\ \lim\limits_{k\to\infty}\frac{N_1(n_1+m_1)+\dots+N_k(n_{k+1}+m_{k+1})}{N_{k+1}}=0$$

$\mathbf{Step\ 1.}$ Constructions of intermediate sets $\{\mathcal{C}_k\}_{k=1}^{\infty}.$

For every $k$ and $\mathcal{S}_k:=\{x_i^k:i=1,\dots,\#\mathcal{S}_k\}$, we consider $\underline{i}=(i_1,\dots,i_{N_k})\in\{1,\dots,\#\mathcal{S}_k\}^{N_k}$. Using the specification property, we can choose a point $y:=y(i_1,\dots,i_{N_k})$ which satisfies
$$d_{n_k}(x_{i_j}^k,f^{a_j}y)<\frac{\epsilon_0}{2^{k+5}},\ for\ j\in\{1,\dots,N_k\},\ a_j=(j-1)(n_k+m_k).$$
We define 
$$\mathcal{C}_k=\{y(i_1,\dots,i_{N_k})\in X:(i_1,\dots,i_{N_k})\in\{1,\dots,\#\mathcal{S}_k\}^{N_k}\}.$$
Denote $c_k=N_kn_k+(N_k-1)m_k$. Then $c_k$ is the amount of time for which the orbit of points in $\mathcal{C}_k$ has been shadowed and we have the following lemma.
\begin{lem}\label{lem 3.9}
	let $\underline{i},\ \underline{j}$ be two different words in $\{1,\dots,\#\mathcal{S}_k\}^{N_k}$. Then $y_1:=y(\underline{i})$ and $y_2:=y(\underline{j})$ are $(c_k,\frac{17\epsilon_0}{16})$-separated points, i.e. $d_{c_k}(y_1,y_2)>{\frac{17\epsilon_0}{16}}$. Especially, $\#\mathcal{C}_k=(\#\mathcal{S}_k)^{N_k}.$
\end{lem}
\begin{proof}
	Since $\underline{i}\neq\underline{j}$, there exists $l$ such that $i_l\neq j_l.$ We have 
	\begin{align*}
		d_{c_k}(y_1,y_2)&\ge d_{n_k}(f^{a_l}y_1,f^{a_l}y_2)\\
&\ge d_{n_k}(x_{i_l}^k,x_{j_l}^k)-d_{n_k}(f^{a_l}y_1,x_{i_l}^k)-d_{n_k}(f^{a_l}y_2,x_{j_l}^k)\\
&>\frac{9\epsilon_0}{8}-\frac{\epsilon_0}{2^{(k+5)}}-\frac{\epsilon_0}{2^{(k+5)}}\\
&\ge\frac{17\epsilon_0}{16}.
	\end{align*}
\end{proof}

$\mathbf{Step\ 2.}$ Constuctions of $\{\mathcal{T}_k\}_{k=1}^{\infty}$, the $k$-th level of the Moran-like fractal.

Let $\mathcal{T}_1=\mathcal{C}_1$ and $t_1=c_1$. We construct $\mathcal{T}_{k+1}$ from $\mathcal{T}_k$ as follows. Let $t_{k+1}:=t_k+m_{k+1}+c_{k+1}$ and $x\in\mathcal{T}_k,\ y\in\mathcal{C}_{k+1}$. By the specification property, we can find a point $z:=z(x,y)$ that satisfies
$$d_{t_k}(x,z)<\frac{\epsilon_0}{2^{(k+6)}}\ and\ d_{c_{k+1}}(y,f^{t_k+m_{k+1}}z)<\frac{\epsilon_0}{2^{(k+6)}}.$$ 
Define $\mathcal{T}_{k+1}=\{z(x,y):x\in\mathcal{T}_k,\ y\in\mathcal{C}_{k+1}\}$, and note that $t_{k+1}$ is the amount of time for which the orbits of points in $\mathcal{T}_k$ has been shadowed. Similarly, we have the following lemma.
\begin{lem}\label{lem 3.10}
	For every $x\in\mathcal{T}_k$ and distinct points $y_1,y_2\in\mathcal{C}_{k+1}$
	$$d_{t_k}(z(x,y_1),z(x,y_2))<\frac{\epsilon_0}{2^{k+5}},\ d_{t_{k+1}}(z(x,y_1),z(x,y_2))\ge\frac{33\epsilon_0}{32}.$$
	Thus, $\mathcal{T}_k$ is a $\left(t_k,\frac{33\epsilon_0}{32}\right)$-separated set. In particular, if $z_1,z_2\in\mathcal{T}_k$, then 
	$$\overline{B}_{t_k}\left(z_1,\frac{\epsilon_0}{2^{k+5}}\right)\cap\overline{B}_{t_k}\left(z_2,\frac{\epsilon_0}{2^{k+5}}\right)=\emptyset.$$
\end{lem}
\begin{proof}
	Let $z_1=z(x,y_1),\ z_2=(x,y_2)$. Hence, we have 
	$$d_{t_k}(z_1,z_2)\leq d_{t_k}(z_1,x)+d_{t_k}(z_2,x)<\frac{\epsilon_0}{2^{(k+6)}}+\frac{\epsilon_0}{2^{(k+6)}}=\frac{\epsilon_0}{2^{(k+5)}}$$
	\begin{align*}
		d_{t_{k+1}}(z_1,z_1)&\ge d_{c_{k+1}}(f^{t_k+m_{k+1}}z_1,f^{t_k+m_{k+1}}z_2)\\
		&\ge\frac{17\epsilon_0}{16}-\frac{\epsilon_0}{2^{(k+6)}}-\frac{\epsilon_0}{2^{(k+6)}}\ge\frac{33\epsilon_0}{32}.
	\end{align*}
	The third statement is a straightforword consequence of the second inequality.
\end{proof}

As a direct result of the Lemma \ref{lem 3.10}, we have 
$$\#\mathcal{T}_k=\#\mathcal{T}_{k-1}\cdot\#\mathcal{C}_k=\#\mathcal{C}_1\dots\#\mathcal{C}_k=\#\mathcal{S}_1^{N_1}\dots\#\mathcal{S}_k^{N_k}.$$
\begin{lem}\label{lem 3.11}
	Let $z=z(x,y)\in\mathcal{T}_k$, then we have
	$$\overline{B}_{t_{k+1}}\left(z,\frac{\epsilon_0}{2^{k+6}}\right)\subset\overline{B}_{t_k}\left(x,\frac{\epsilon_0}{2^{k+5}}\right).$$
\end{lem}
\begin{proof}
	From the constructin, $d_{t_k}(z,x)<\frac{\epsilon_0}{2^{k+6}}.$ Thus, for any point $p\in\overline{B}_{t_{k+1}}(z,\frac{\epsilon_0}{2^{k+6}})$, one has
	$$d_{t_k}(p,x)\leq d_{t_k}(p,z)+d_{t_k}(z,x)\leq \frac{\epsilon_0}{2^{k+6}}\cdot2\leq\frac{\epsilon_0}{2^{k+5}}$$
	which implies that $p\in\overline{B}_{t_k}(x,\frac{\epsilon_0}{2^{k+5}})$. Therefore, the result has been proved.
\end{proof}

$\mathbf{Step\ 3.}$ Constructions of the Moran-like fractal contained in $K_{\alpha}.$

Let $F_k=\cup_{x\in\mathcal{T}_k}\overline{B}_{t_k}(x,\frac{\epsilon_0}{2^{k+5}})$. By Lemma \ref{lem 3.11}, $F_{k+1}\subset F_k$ and we have a decreasing sequence of compact sets, the set $F=\cap_kF_k$ is non-empty. Besides, every point $p\in F$ can be uniquely represented by a sequence $\underline{p}=(\underline{p}_1,\underline{p}_2,\dots)$, where each $\underline{p}_i=(\underline{p}_1^i,\dots,\underline{p}_{N_i}^i)\in\{1,2,\dots,\#\mathcal{S}_i\}^{N_i}.$ Thus, every point in $\mathcal{T}_k$ can be uniquely represented by a finite word $\underline{p}=(\underline{p}_1,\dots,\underline{p}_k).$
\begin{lem}\label{lem 3.12}
	Given $z=z(\underline{p}_1,\dots,\underline{p}_k)\in\mathcal{T}_k$, for all $i\in\{1,\dots,k\}$ and all $l\in\{1,\dots,N_i\}$ we have that
	$$d_{n_i}(x_{p_l^i}^i,f^{t_{i-1}+m_i+(l-1)(m_i+n_i)}z)<\epsilon_0.$$
\end{lem}
\begin{proof}
Given $i\in\{1,\dots,k\}$ and $l\in\{1,\dots,N_i\}.$ For $m\in\{1,\dots,k-1\}$, let $z_m=z(\underline{p}_1,\dots,\underline{p}_m)\in \mathcal{T}_m$. Let $a=t_{i-1}+m_i,\ b=(l-1)(m_i+n_i)$. Then
	\begin{align*}
	d_{n_i}(x_{p_l^i}^i,f^{a+b}z)&\leq d_{n_i}(x_{p_l^i}^i,f^{b}y{\underline{p}_i})+d_{n_i}(f^by_{\underline{p}_i}^i,f^{a+b}z)+d_{n_i}(f^{a+b}z_i,f^{a+b}z)\\
&<\frac{\epsilon_0}{2^{i+5}}+d_{c_i}(y_{\underline{p}_i}^i,f^az)+d_{t_i}(z_i,z)\\
&<\frac{\epsilon_0}{2^{i+5}}+\frac{\epsilon_0}{2^{i+6}}+d_{t_i}(z_i,z_{i+1})+\cdots+d_{t_i}(z_{k-1},z_k)\\
&<\frac{\epsilon_0}{2^{i+5}}+\frac{\epsilon_0}{2^{i+6}}+\frac{\epsilon_0}{2^{i+6}}+\frac{\epsilon_0}{2^{i+7}}+\cdots+\frac{\epsilon_0}{2^{k+5}}\\
&<\sum_{m=1}^{k}\frac{\epsilon_0}{2^{m+5}}+\frac{\epsilon_0}{2^{i+6}}<\epsilon_0.
\end{align*}
\end{proof}

\begin{lem}\label{lem 3.13}
	Under the above conditions, $F\subset K_{\alpha}.$
\end{lem}
\begin{proof}
	For any $x\in F$, we only need to show $\lim\limits_{n\to\infty}\left|\frac{1}{n}S_n\varphi(x)-\alpha\right|=0.$ Thus, we need to estimate $\left|S_n\varphi(x)-n\alpha\right|$. We can divide the estimation into 3 steps.
	
	$\mathbf{Step\ 1.}$ Estimation on $\mathcal{C}_k$ for $k\ge1.$

	Supposing $y\in\mathcal{C}_k$, let us estimate $\left|\sum_{p=0}^{c_k-1}\varphi(f^py)-c_k\alpha\right|$. By the construction of $\mathcal{C}_k$, there exists $(i_1,\dots,i_{N_k})\in(1,\dots,\#\mathcal{S}_k)^{N_k}$  and $x_{i_j}^k\in\mathcal{S}_k $ satisfying
	$$d_{n_k}(x_{i_j}^k,f^{a_j}y)<\frac{\epsilon_0}{2^{k+5}}\ for\ j=1,\dots,N_k.$$
	Since 
	\begin{align*}
	[0,c_k-1]&=[0,N_kn_k+(N_k-1)m_k-1]\\
	&=\bigcup_{j=1}^{N_k}[a_j,a_j+n_k-1]\cup\bigcup_{j=1}^{N_k-1}[a_j+n_k,a_j+n_k+m_k-1]
\end{align*}
On $[a_j,a_j+n_k-1]$, we have 
\begin{align*}
\left|\sum_{p=0}^{n_k-1}\varphi(f^{a_j+p}y)-n_k\alpha\right|\leq &\left|\sum_{p=0}^{n_k-1}\varphi(f^{a_j+p}y)-\sum_{p=0}^{n_k-1}\varphi(f^px_{i_j}^k)\right|\\
&+\left|\sum_{p=0}^{n_k-1}\varphi(f^px_{i_j}^k)-n_k\alpha\right|\\
&\leq n_k\Big(Var\left(\varphi,\frac{\epsilon_0}{2^{k+5}}\right)+2\delta_k+Var(\varphi,\frac{\epsilon_0}{2^{k+5}})+\frac{1}{k}\Big).
\end{align*}
On $[a_j+n_k,a_j+n_k+m_k-1]$, we have
$$\left|\sum_{p=0}^{m_k-1}\varphi(f^{a_j+n_k+p}y)-m_k\alpha\right|\leq m_k(\|\varphi\|+\alpha)\leq 2m_k\|\varphi\|.$$
Combining these inequalities, we have
\begin{align*}
&\left|\sum_{p=0}^{c_k-1}\varphi(f^py)-c_k\alpha\right|\\
\leq& N_kn_k\Big(2Var\left(\varphi,\frac{\epsilon_0}{2^{k+5}}\right)+2\delta_k+\frac{1}{k}\Big)+2(N_k-1)m_k\|\varphi\|.
\end{align*}
$\mathbf{Step\ 2.}$ Estimation on $\mathcal{T}_k$ for $k\ge2.$

For $k\ge2$, let us estimate 
$$A_k:=\max_{x\in\mathcal{T}_k}\left|\sum_{p=0}^{t_k-1}\varphi(f^pz)-t_k\alpha\right|.$$
For any $z\in\mathcal{T}_k$, there exists $x\in\mathcal{T}_{k-1}$ and $y\in\mathcal{C}_k$ satisfying
$$d_{t_{k-1}}(x,z)<\frac{\epsilon_0}{2^{k+5}},\ d_{c_k}(y,f^{t_k+m_{k-1}}z)<\frac{\epsilon_0}{2^{k+5}}$$
On $[0,t_{k-1}+m_k-1]$, we have $|\varphi-\alpha|\leq 2\|\varphi\|$, while on $[t_{k-1}+m_k,t_k-1],$ we use the specification property and estimation on $\mathcal{C}_k$ to obtain
\begin{align*}
	A_k\leq&2(t_{k-1}+m_k)\|\varphi\|+c_kVar(\varphi,\frac{\epsilon_0}{2^{k+5}})\\
	&+N_kn_k\Big(2Var\left(\varphi,\frac{\epsilon_0}{2^{k+5}}\right)+2\delta_k+\frac{1}{k}\Big)+2(N_k-1)m_k\|\varphi\|.
\end{align*}
And by the choice of $N_k\ and\ n_k$, we have
$$\frac{t_{k-1}+m_k}{N_k}\to0,\ \frac{(N_k-1)m_k}{N_kn_k}\to0\ \text{as}\ k\to\infty.$$
Since $t_k$, taking $k\to\infty$ we have 
$$\frac{A_k}{t_k}\leq\frac{2(t_{k-1}+m_k)\|\varphi\|}{N_k}+3Var\left(\varphi,\frac{\epsilon_0}{2^{k+5}}\right)+2\delta_k+\frac{2(N_k-1)m_k\|\varphi\|}{N_kn_k}+\frac{1}{k}\to0.$$

$\mathbf{Step\ 3.}$ Estimation on $F$.

For any $x\in F,\ n>t_1$, there exists unique $k\ge1\ and\ 0\leq j\leq N_{k+1}-1$ such that $t_k+j(n_{k+1}+m_{k+1})<n\leq t_k+(j+1)(n_{k+1}+m_{k+1}).$ On one hand, since $x\in F$, there exists $z\in\mathcal{T}_k$ such that $d_{t_{k+1}}(x,z)<\frac{\epsilon_0}{2^{k+6}},$ on the other hand, exist $\overline{x}\in\mathcal{T}_k\ and\ y\in\mathcal{C}_{k+1}$ satisfying
$$d_{t_k}(\overline{x},z)<\frac{\epsilon_0}{2^{k+6}},\ d_{c_{k+1}}(y,f^{t_k+m_{k+1}}z)<\frac{\epsilon_0}{2^{k+6}}.$$
Hence, we have
$$d_{t_k}(\overline{x},x)<\frac{\epsilon_0}{2^{k+5}},\ d_{c_{k+1}}(y,f^{t_k+m_{k+1}}x)<\frac{\epsilon_0}{2^{k+5}}.$$
Furthermore, by the constructin of $\mathcal{C}_{k+1}$, there exists $x_{i_1}^{k+1},\dots,x_{i_j}^{k+1}\in\mathcal{S}_{k+1}$ such that for $t=1,\dots,j$ we have
\begin{align*}
d_{n_{k+1}}(x_{i_t}^{k+1},f^{t_k+m_{k+1}+a_t}x)&\leq d_{n_{k+1}}(x_{i_t}^{k+1},f^{a_t}y)+d_{n_{k+1}}(f^{a_t}y,f^{t_k+m_{k+1}+a_t}x)\\
&<\frac{\epsilon_0}{2^{k+5}}+\frac{\epsilon_0}{2^{k+5}}=\frac{\epsilon_0}{2^{k+4}},
\end{align*}
where $a_t=(n_{k+1}+m_{k+1})(t-1).$ Besides, 
$$[0,n-1]=[0,t_k-1]\cup\bigcup_{t=1}^{j}[t_k+a_t,t_k+a_{t+1}-1]\cup[t_k+j(m_{k+1}+n_{k+1}),n-1].$$
On $[0,t_k-1]$, we have
\begin{align*}
	\left|\sum_{p=0}^{t_k-1}\varphi(f^px)-t_k\alpha\right|&\leq\left|\sum_{p=0}^{t_k-1}\varphi(f^px)-\sum_{p=0}^{t_k-1}\varphi(f^p{\overline{x}})\right|+\left|\sum_{p=0}^{t_k-1}\varphi(f^px)-t_k\alpha\right|\\
	&\leq t_k\cdot Var\left(\varphi,\frac{\epsilon_0}{2^{k+5}}\right)+A_k,
\end{align*}
On $[t_k+a_t,t_k+a_{t+1}-1]$, we have
\begin{align*}
	&\left|\sum_{p=t_k+a_t}^{t_k+a_{t+1}-1}\varphi(f^px)-(m_{k+1}+n_{k+1})\alpha\right|\\
	\leq&\left|\sum_{p=t_k+a_t}^{t_k+a_{t+1}-1}\varphi(f^px)-\sum_{p=t_k+a_t+m_{k+1}}^{t_k+a_{t+1}-1}\varphi(f^px)\right|\\
	&+\left|\sum_{p=t_k+a_t+m_{k+1}}^{t_k+a_{t+1}-1}\varphi(f^px)-\sum_{p=0}^{n_{k+1}-1}\varphi(f^px_{i_j}^{k+1})\right|\\
	&+\left|\sum_{p=0}^{n_{k+1}-1}\varphi(f^px_{i_j}^{k+1})-n_{k+1}\alpha\right|+m_{k+1}|\alpha|\\
	\leq&2m_{k+1}\|\varphi\|+n_{k+1}\cdot Var\left(\frac{\epsilon_0}{2^{k+4}}\right)+n_{k+1}\Big(2\delta_{k+1}+Var(\varphi,\frac{\epsilon_0}{2^{k+6}})+\frac{1}{k+1}\Big).
\end{align*}
Finally, on $[t_k+j(m_{k+1}+n_{k+1}),n-1]$, we have
\begin{align*}
&\left|\sum_{p=t_k+j(n_{k+1}+m_{k+1})}^{n-1}\varphi(f^px)-(n-t_k-j(n_{k+1}+m_{k+1}))\alpha\right|\\
&\leq2(n-t_k-j(n_{k+1}+m_{k+1}))\|\varphi\|\leq2(n_{k+1}+m_{k+1})\|\varphi\|.
\end{align*}
Combining the above three estimation, we have
\begin{align*}
	&\left|\sum_{p=0}^{n-1}\varphi(f^px)-n\alpha\right|\leq t_k\cdot Var\left(\varphi,\frac{\epsilon_0}{2^{k+5}}\right)+A_k+2(n_{k+1}+m_{k+1})\|\varphi\|\\
&+j\left(2m_{k+1}\|\varphi\|+n_{k+1}\cdot Var\left(\frac{\epsilon_0}{2^{k+4}}\right)+n_{k+1}\Big(2\delta_{k+1}+Var(\varphi,\frac{\epsilon_0}{2^{k+6}})+\frac{1}{k+1}\Big)\right).\\
\end{align*}
Hence,
\begin{align*}
	\left|\frac{1}{n}\sum_{p=0}^{n-1}\varphi(f^px)-\alpha\right|\leq&\frac{A_k}{t_k}+Var\left(\varphi,\frac{\epsilon_0}{2^{k+4}}\right)\\
&+2\left(\frac{n_{k+1}+m_{k+1}}{N_k}+\frac{m_{k+1}}{n_{k+1}}\right)\|\varphi\|+\frac{1}{k+1}.
\end{align*}
Letting $n\to\infty$, $k\to\infty$ and the right side $\to0$.
The proof has been completed.
\end{proof}

We now define the measure on $F$ which satisfies the pressure distribution principle.
For each $k$ and every $z=z(\underline{p}_1,\dots,\underline{p}_k)\in\mathcal{T}_k$, we define $\mathcal{L}(z):=\mathcal{L}(\underline{p}_1)\dots\mathcal{L}(\underline{p}_k)$ and $\underline{p}_i=(p_1^i,\dots,p_1^{N_i})\in\{1,\dots,\#\mathcal{S}_i\}^{N_i}$. Let
$$\mathcal{L}(\underline{p}_i):=\prod_{l=1}^{N_i}\exp\left\{S_{n_i}\psi(x_{p_l^i}^i)\log{\frac{1}{5\epsilon_0}}\right\},\ \nu_k:=\sum_{x\in\mathcal{T}_k}^{}\delta_z\mathcal{L}(z).$$ 
Normalizing $\nu_k$ to obtain a sequence of probability measures $\mu_k$. We let
$$\kappa_k:=\sum_{x\in\mathcal{T}_k}^{}\mathcal{L}(z),\ \mu_k:=\frac{1}{\kappa_k}\nu_k$$
\begin{lem}\label{lem 3.14}
	$\kappa_k=\prod_{i=1}^{k}M_i^{N_i}.$
\end{lem}
\begin{proof}
	We note that
	\begin{align*}
		\kappa_k=\sum_{x\in\mathcal{T}_k}^{}\mathcal{L}(z)&=\sum_{\underline{p}_1\in\{1,\dots,\#\mathcal{S}_1\}^{N_1}}^{}\dots\sum_{\underline{p}_k\in\{1,\dots,\#\mathcal{S}_k\}^{N_k}}^{}(\mathcal{L}(\underline{p}_1)\cdots\mathcal{L}(\underline{p}_1))\\
		&=\left(\sum_{\underline{p}_1\in\{1,\dots,\#\mathcal{S}_1\}^{N_1}}^{}\mathcal{L}(\underline{p}_1)\right)\cdots\left(\sum_{\underline{p}_1\in\{1,\dots,\#\mathcal{S}_k\}^{N_k}}^{}\mathcal{L}(\underline{p}_k)\right).
	\end{align*}
From the definition of $\mathcal{L}(\underline{p}_i)$, for every $i$ we have
$$\left(\sum_{\underline{p}_i\in\{1,\dots,\#\mathcal{S}_i\}^{N_i}}^{}\mathcal{L}(\underline{p}_1)\right)=\prod_{l=1}^{N_i}\left\{\sum_{p^i_l=1}^{\#\mathcal{S}_i}\exp\left\{S_{n_i}\psi(x_{p_l^i}^i)\log{\frac{1}{5\epsilon_0}}\right\}\right\}=M_i^{N_i}$$
Hence, $\kappa_k=\prod_{i=1}^{k}M_i^{N_i}.$
\end{proof}
\begin{lem}\cite[Lemma 3.10]{th10} \label{lem 3.15}
	Suppose $\nu$ is an accumulation point of the sequence of probability measures $\mu_k$. Then $\nu(F)=1$.
\end{lem}
\begin{proof}
Let $\nu=\lim\limits_{k\to\infty}\mu_{l_k}$ for some $l_k\to\infty$. For any fixed $l$ and all $p\ge0$, since $\mu_{l+p}(F_{l+p})=1$ and $F_{l+p}\subset F_l$, we have $\mu_{l+p}(F_l)=1$. Thus, $\nu(F_l)\ge\limsup_{k\to\infty}\mu_{l_k}(F_{l})=1$. It implies that 
$$\nu(F)=\lim\limits_{l\to\infty}\nu(F_l)=1.$$
\end{proof}

Let $\mathcal{B}:=B_n(q,{\epsilon_0/2})$ be an arbitrary ball which intersects $F$. There exists unique $k$ that satisfies $t_k\leq n<t_{k+1}$ and unique $j\in\{0,\dots,N_{k+1}-1\}$ that satisfies $t_k+j(n_{k+1}+m_{k+1})\leq n<t_k+(j+1)(n_{k+1}+m_{k+1}).$ Thus, we have the following lemma which reflects the fact that the number of points in $\mathcal{B}\cap\mathcal{T}_{k+1}$ is restricted.
\begin{lem}\label{lem 3.16}
	Suppose $\mu_{k+1}(\mathcal{B})>0$, then there exists a unique $x\in\mathcal{T}_k$ and $i_1,\dots,i_j\in\{1,\dots,\#\mathcal{S}_{k+1}\}$ satisfying
	$$\nu_{k+1}(\mathcal{B})\leq\mathcal{L}(x)\left(\prod_{l=1}^{j}\exp\left(S_{n_{k+1}}\psi(x_l^{k+1})\log{\frac{1}{5\epsilon_0}}\right)\right)M_{k+1}^{n_{k+1}-j}.$$
\end{lem}
\begin{proof}	
	Since we suppose $\mu_{k+1}(\mathcal{B})>0$, then $\mathcal{T}_{k+1}\cap\mathcal{B}\neq\emptyset.$ Let $z_1=z(x_1,y_1),\ z_2=z(x_2,y_2)\in\mathcal{T}_k\cap\mathcal{B}$, where $x_1,x_2\in\mathcal{T}_k\ and\ y_1,y_2\in\mathcal{C}_{k+1}.$ Let $y_1=y(i_1,\dots,i_{N_{k+1}}),\ y_2=y(l_1,\dots,l_{N_{k+1}}).$ We have
	\begin{align*}
	d_{t_k}(x_1,x_2)&\leq d_{t_k}(x_1,z_1)+d_{t_k}(z_1,z_2)+d_{t_k}(z_2,x_2)\\
&<\frac{\epsilon_0}{2^{k+6}}+\epsilon_0+<\frac{\epsilon_0}{2^{k+6}}<\frac{33\epsilon_0}{32}
\end{align*}
and thus we have a contradiction with the fact that $\mathcal{T}_k$ is $(t_k,\frac{33\epsilon_0}{32})$ separated. Similarly, we prove that $i_t=l_t$ for $t=\{1,\dots,j\}.$ Suppose there exists $t,\ 1\leq t\leq j$, such that $i_t\neq l_t$. By the specification property, we have 
$$d_{n_{k+1}}(x_{i_t}^{k+1},f^{a_t}y_1)<\frac{\epsilon_0}{2^{k+5}},\ d_{n_{k+1}}(x_{l_t}^{k+1},f^{a_t}y_2)<\frac{\epsilon_0}{2^{k+5}}$$
and
$$d_{c_{k+1}}(y_1,f^{t_k+m_{k+1}}z_1)<\frac{\epsilon_0}{2^{k+6}},\ d_{c_{k+1}}(y_1,f^{t_k+m_{k+1}}z_1)<\frac{\epsilon_0}{2^{k+6}}.$$
Thus, 
\begin{align*}
	d_{n_{k+1}}(x_{i_t}^{k+1},x_{l_t}^{k+1})&\leq d_{n_{k+1}}(x_{i_t}^{k+1},f^{a_t}y_1)+d_{c_{k+1}}(y_1,f^{t_k+m_{k+1}}z_1)\\
	&+d_n(z_1,z_2)+d_{c_{k+1}}(y_2,f^{t_k+m_{k+1}}z_2)+d_{n_{k+1}}(x_{l_t}^{k+1},f^{a_t}y_2)\\
	&<\epsilon_0+\frac{\epsilon_0}{2^{k+5}}\cdot4<\frac{9\epsilon_0}{8}
\end{align*}
which contradicts the fact that $\mathcal{S}_{k+1}$ is $\left(n_{k+1},\frac{9\epsilon_0}{8}\right)$ separated.

Since $x$ and $(i_1,\dots,i_j)$ is the same for all points $z=z(x,y)$, $y=y(i_1,\dots,i_{N_{k+1}})$ which lies in $\mathcal{T}_{k+1}\cap\mathcal{B}$, we can conclude that there are at most $M_{k+1}^{N_{k+1}-j}$ such points. Hence,
\begin{align*}
\nu_{k+1}(\mathcal{B})&\leq\mathcal{L}(x)\sum_{\underline{p}_{k+1}}^{}\mathcal{L}(\underline{p}_{k+1})\\
&\leq\mathcal{L}(x)\left(\prod_{l=1}^{j}\exp\left(S_{n_{k+1}}\psi(x_l^{k+1})\log{\frac{1}{5\epsilon_0}}\right)\right)M_{k+1}^{n_{k+1}-j}.
\end{align*}
\end{proof}

\begin{lem}\label{lem 3.17}
	Let $x\in\mathcal{T}_k$ and $i_1,\dots,i_j$ be as before. Then 
	\begin{align*}
		&\mathcal{L}(x)\prod_{l=1}^{j}\exp\left(S_{n_k+1}\psi(x_{i_l}^{k+1})\cdot\log\frac{1}{5\epsilon_0}\right)\\
	&\leq\exp\left(\left(S_n\psi(q)+2nVar(\psi,\epsilon_0)+\left(jm_{k+1}+\sum_{i=1}^{k}N_im_i\right)\|\psi\|\right)\log{\frac{1}{5\epsilon_0}}\right).
	\end{align*}
\end{lem}
\begin{proof}
	Let $x:=x(\underline{p}_1,\dots,\underline{p}_k)$. It follows from Lemma \ref{lem 3.13} that
	$$d_{n_i}(x_{p_l^i}^i,f^{t_{i-1}+m_i+(l-1)(m_i+n_i)}x)<\epsilon_0.$$
	for $i\in\{1,\dots,k\}$ and $l\in\{1,\dots,N_i\}$ and
	\begin{align*}
		S_{n_i}\psi(x_{p_l^i}^i)\leq& S_{n_i}\psi(x_{p_l^i}^i)-S_{n_i}\psi(f^{t_{i-1}+m_i+(l-1)(m_i+n_i)}x)\\
		&+S_{n_i}\psi(f^{t_{i-1}+m_i+(l-1)(m_i+n_i)}x)\\
		\leq& n_iVar(\psi,\epsilon_0)+S_{n_i}\psi(f^{t_{i-1}+m_i+(l-1)(m_i+n_i)}x)
	\end{align*}
	and
	\begin{align*}
		S_{n_{k+1}}\psi(x_{i_l}^{k+1})\leq& S_{n_{k+1}}\psi(x_{i_l}^{k+1})-S_{n_{k+1}}\psi(f^{t_k+m_{k+1}+(l-1)(n_{k+1}+m_{k+1})}z)\\
		&+S_{n_{k+1}}\psi(f^{t_k+m_{k+1}+(l-1)(n_{k+1}+m_{k+1})}z)\\
		\leq&n_{k+1}Var\left(\varphi,\frac{\epsilon_0}{2^{k+5}}\right)+S_{n_{k+1}}\psi(f^{t_k+m_{k+1}+(l-1)(n_{k+1}+m_{k+1})}z).
	\end{align*}
Thus we have
\begin{align*}
&\sum_{i=1}^{k}\sum_{l=1}^{N_i}S_{n_i}\psi(x_{p_l^i}^i)\cdot\log{\frac{1}{5\epsilon_0}}\\
\leq&\left\{\sum_{i=1}^{k}\sum_{l=1}^{N_i}n_iVar(\psi,\epsilon_0)+\sum_{i=1}^{k}\sum_{l=1}^{N_i}S_{n_i}\psi(f^{t_{i-1}+m_i+(l-1)(m_i+n_i)}x)\right\}\cdot\log{\frac{1}{5\epsilon_0}}\\
\leq&\left\{t_kVar(\psi,\epsilon_0)+S_{t_{k}}\psi(x)+\sum_{i=1}^{k}N_im_i\|\psi\|\right\}\cdot\log{\frac{1}{5\epsilon_0}}
\end{align*}
and
\begin{align*}
	\sum_{l=1}^{j}S_{n_{k+1}}\psi(x_{i_l}^{k+1})\leq(n-t_k)Var\left(\psi,\frac{\epsilon_0}{2^{k+6}}\right)+S_{n-t_k}\psi(f^{t_k}z)+jm_{k+1}\|\psi\|.
\end{align*}
Since $d_{t_k}(x,q)\leq d_{t_k}(x,z)+d_{t_k}(q,z)<\epsilon_0$, we have
\begin{align*}
S_{t_{k}}\psi(x)+S_{n-t_k}\psi(f^{t_k}z)=&S_{t_{k}}\psi(x)-S_{t_{k}}\psi(q)+S_{n-t_k}\psi(f^{t_k}z)\\
&-S_{n-t_k}\psi(f^{t_k}q)+S_n\psi(q)\\
\leq &t_kVar(\psi,\epsilon_0)+(n-t_k)Var(\psi,\epsilon_0)+S_n\psi(q).
\end{align*}
Combining the above inequalities, we have that
\begin{align*}
	&\mathcal{L}(x)\prod_{l=1}^{j}\exp\left(S_{n_k+1}\psi(x_{i_l}^{k+1})\cdot\log\frac{1}{5\epsilon_0}\right)\\
&\leq\exp\left\{\left(S_n\psi(q)+2nVar(\psi,\epsilon_0)+\left(jm_{k+1}+\sum_{i=1}^{k}N_im_i\right)\|\psi\|\right)\log{\frac{1}{5\epsilon_0}}\right\}.
\end{align*}
\end{proof}
Similarly, we give the following Lemma for  the points contained in $\mathcal{B}\cap\mathcal{T}_{k+p} $ without proof.
\begin{lem}
	For any $p\ge1$, suppose $\mu_{k+p}(\mathcal{B})>0$. Let $x\in\mathcal{T}_k$ and $i_1,\dots,i_j$ be as before. Then every $x\in\mathcal{B}\cap\mathcal{T}_{k+p}$ descends from some point in $\mathcal{T}_k\cap\mathcal{B}$. We have
	$$\nu_{k+p}(\mathcal{B})\leq\mathcal{L}(x)\left\{\prod_{l=1}^{j}\exp\left(S_{n_{k+1}}\psi(x_l^{k+1})\log{\frac{1}{5\epsilon_0}}\right)\right\}M_{k+1}^{n_{k+1}-j}\cdots M_{k+p}^{n_{k+p}}.$$
\end{lem}
Since $\mu_{k+p}=\frac{1}{\kappa_{k+p}}\nu_{k+p}$ and $\kappa_{k+p}=\kappa_kM_{k+1}^{N_{k+1}}\cdots M_{k+p}^{N_{k+p}}$, immediately, we have 
\begin{align*}
&\mu_{k+p}(\mathcal{B})\leq\frac{1}{\kappa_kM_{k+1}^j}\mathcal{L}(x)\left\{\prod_{l=1}^{j}\exp\left(S_{n_{k+1}}\psi(x_l^{k+1})\log{\frac{1}{5\epsilon_0}}\right)\right\}\\
&\leq\frac{1}{\kappa_kM_{k+1}^j}\exp\left\{\left(S_n\psi(q)+2nVar(\psi,\epsilon_0)+\left(jm_{k+1}+\sum_{i=1}^{k}N_im_i\right)\|\psi\|\right)\log{\frac{1}{5\epsilon_0}}\right\}.
\end{align*}

\begin{lem}
	For suffciently large $n$, we have 
	$$\kappa_kM_{k+1}^j\ge\exp\left(\left(\inf_{diam\xi<5\epsilon_0}h_{\mu}(f,\xi)+\Big(\log{\frac{1}{5\epsilon_0}}\Big)\int\psi \mathrm{d}{\mu}\cdot-5\gamma\right)n\right).$$
\end{lem}
\begin{proof}
	From Lemma \ref{lem 3.8}, we have
	$$M_k\ge\exp\left\{n_k\left(\inf_{diam\xi<5\epsilon_0}h_{\mu}(f,\xi)+\Big(\log{\frac{1}{5\epsilon_0}}\Big)\int\psi \mathrm{d}{\mu}-4\gamma\right)\right\}.$$
	Taking a note that $C:=\inf_{diam\xi<5\epsilon_0}h_{\mu}(f,\xi)+\Big(\log{\frac{1}{5\epsilon_0}}\Big)\int\psi \mathrm{d}{\mu}$ and we have
	\begin{align*}
		&\kappa_kM_{k+1}^j\\
		&=M_1^{N_1}\cdots M_k^{N_k}\ge\exp\left((C-4\gamma)(n_1N_1+\cdots+n_kN_k+n_{k+1}j)\right)\\
		&\ge\exp\left((C-5\gamma)((n_1+m_1)N_1+\cdots+(n_k+m_k)N_k+(n_{k+1}+m_{k+1})(j+1))\right)\\
		&=\exp\left((C-5\gamma)(t_k+(n_{k+1}+m_{k+1})(j+1))\right)\\
		&\ge\exp\left((C-5\gamma)n\right).
	\end{align*}
\end{proof}
\begin{lem}\label{lem 3.20}
	For suffciently large $n$, we have 
	$$\limsup_{k\to\infty}\mu_k\left(B_n\left(q,\frac{\epsilon_0}{2}\right)\right)\leq\exp\left\{-n(C-6\gamma-Var(\psi,\epsilon_0))+S_n\psi(q)\cdot\log{\frac{1}{5\epsilon_0}}\right\}.$$
\end{lem}
\begin{proof}
	For suffciently large $n$, and any $p>1$,
	\begin{align*}
		&\mu_{k+p}(\mathcal{B})\\
		&\leq\frac{1}{\kappa_kM_{k+1}^j}\exp\left\{\left(S_n\psi(q)+2nVar(\psi,\epsilon_0)+\left(jm_{k+1}+\sum_{i=1}^{k}N_im_i\right)\|\psi\|\right)\log{\frac{1}{5\epsilon_0}}\right\}\\
		&\leq\frac{1}{\kappa_kM_{k+1}^j}\exp\left\{S_n\psi(q)\cdot\log{\frac{1}{5\epsilon_0}}+n(2Var(\psi,\epsilon_0)+\gamma)\right\}\\
		&\leq\exp\left\{-n(C-5\gamma)+S_n\psi(q)\cdot\log{\frac{1}{5\epsilon_0}}+n(2Var(\psi,\epsilon_0)+\gamma)\right\}\\
		&=\exp\left\{-n(C-6\gamma-2Var(\psi,\epsilon_0))+S_n\psi(q)\cdot\log{\frac{1}{5\epsilon_0}}\right\}.
	\end{align*}
	We arrive the sencond inequality is because $n_k\gg m_k$, thus
	$$\frac{jm_{k+1}+\sum_{i=1}^{k}N_im_i}{n}\leq\frac{jm_{k+1}+\sum_{i=1}^{k}N_im_i}{t_k+j(n_{k+1}m_{k+1})}\to0,\ as\ k\to\infty.$$
\end{proof}
Now we give the generalized pressure distribution principle which is a modification of \cite[Proposition 3.2]{th10}.
\begin{lem}\label{lem 3.21}
	Let $f:X\to X$ be a continuous transformation and $\epsilon>0$. For $Z\subset X$ and a constant $s\ge0$, suppose there exist a constant $C>0$, a sequence of Borel probability measure $\mu_k$ and integer $N$ satisfying
	$$\limsup_{k\to\infty}\mu_k\left(B_n\left(x,\frac{\epsilon}{2}\right)\right)\leq C\exp\left\{-sn+S_n\psi(x)\log{\frac{1}{5\epsilon}}\right\}$$
	for every $B_n(x,\epsilon/2)$ such that $B_n\left(x,\epsilon/2\right)\cap Z\neq\emptyset$ and $n\ge N$. Furthermore, assume that at least one accumulate point $\nu$ of $\mu_k$ satisfies $\nu(Z)>0$. Then $M_{\epsilon/2}\left(f,Z,d,\psi\right)\ge s.$
\end{lem}
\begin{proof}
	Let  $\mu_{k_j}$ be the sequence of measures which converges to $\nu$. Let $\Gamma=\{B_{n_i}(x_i,\epsilon/2)\}_{i\in I}$ cover $Z$ with  $n_i\ge N$. We can assume that $B_{n_i}(x,\epsilon/2)\cap Z\neq\emptyset$ for every $i.$ Then
	\begin{align*}
		\sum_{i\in I}^{}\exp\left\{-sn_i+S_{n_i}\psi(x_i)\log{\frac{2}{\epsilon}}\right\}&\ge\sum_{i\in I}^{}\exp\left\{-sn_i+S_{n_i}\psi(x_i)\log{\frac{1}{5\epsilon}}\right\}\\
		&\ge\frac{1}{C}\sum_{i\in I}^{}\limsup_{k\to\infty}\mu_k\left(B_{n_i}(x_i,\epsilon)\right)\\
		&\ge\frac{1}{C}\sum_{i\in I}^{}\liminf_{j\to\infty}\mu_{k_j}\left(B_{n_i}(x_i,\epsilon)\right)\\
&\ge\frac{1}{C}\sum_{i\in I}^{}\nu\left(B_{n_i}(x_i,\epsilon)\right)\\
&\ge\frac{1}{C}\nu(Z)>0.
	\end{align*}
	Thus, we conclude that $m_{\epsilon/2}(f,Z,s,d,\psi)>0$ and $M_{\epsilon/2}(f,Z,d,\psi)\ge s.$
\end{proof}
By Lemma \ref{lem 3.13} we have that $M_{\epsilon_0}(f,K_{\alpha},d,\psi)\ge M_{\epsilon_0}(f,F,d,\psi)$. By Lemma \ref{lem 3.20} and Lemma \ref{lem 3.21}, we have that 
$$\inf_{diam\xi<5\epsilon_0}h_{\mu}(f,\xi)+\Big(\log{\frac{1}{5\epsilon_0}}\Big)\int\psi \mathrm{d}{\mu}\leq M_{\epsilon_0/2}(f,K_{\alpha},d,\psi)+6\gamma+Var(\psi,\epsilon_0).$$
Combining this inequality with (\ref{3.2}), we have that
\begin{align*}
	S-\gamma&\leq\frac{\inf_{diam\xi<5\epsilon_0}h_{\mu}(f,\xi)+\Big(\log{\frac{1}{5\epsilon_0}}\Big)\int\psi \mathrm{d}{\mu}}{\log{\frac{1}{5\epsilon_0}}}\\
	&\leq\frac{M_{\epsilon_0/2}(f,K_{\alpha},d,\psi)+6\gamma+Var(\psi,\epsilon_0)}{\log{\frac{2}{\epsilon_0}}}\cdot\frac{\log{\frac{2}{\epsilon_0}}}{\log{\frac{1}{5\epsilon_0}}}.
\end{align*}
As $\gamma>0$ is arbitrary and $\gamma\to0\Rightarrow\epsilon_0\to0$, we obtain  
$$S\leq\overline{\rm{mdim}}_M^B(f,K_{\alpha},d,\psi).$$ 
That is, we have 
$${\rm{H}}_{\varphi}\overline{\rm{mdim}}_M(f,K_{\alpha},d,\psi)\leq\overline{\rm{mdim}}_M^B(f,K_{\alpha},d,\psi).$$
\end{proof}
If without the initially assumption that $\overline{\rm{mdim}}_M^B(f,K_{\alpha},d,\psi)$, \\$\overline{\rm{mdim}}_M^P(f,K_{\alpha},d,\psi)$
$\Lambda_{\varphi}\overline{\rm{mdim}}_M(f,K_{\alpha},d,\psi)$ and ${\rm{H}}_{\varphi}\overline{\rm{mdim}}_M(f,K_{\alpha},d,\psi)$ are finite, we slightly modify
our proof to show that any one of the quantities is infinite then the other three quantities must be infinite. Therefore, Theorem \ref{thm 2.3} is still valid.

\section*{Acknowledgement} 
The  second author was supported by the
National Natural Science Foundation of China (No. 12071222).
The third author  was  supported by the
National Natural Science Foundation of China (No. 11971236), Qinglan Project of Jiangsu Province of China. 

\section*{Data availability} 
No data was used for the research described in the article.
\section*{Conflict of interest} 
The author declares no conflict of interest.

\end{document}